\documentclass[onecolumn,nonote,noversion]{cdcarticle}

\def\MODE{1}
\usepackage{math}
\usepackage[hidelinks]{hyperref}
\usepackage{enumerate}
\usepackage{tikz}
\usepackage{caption,subcaption}

\if\MODE1\graphicspath{{./graphics_color/}}
\else\graphicspath{{./graphics_bw_eps/}}\fi

\newcommand{\cvx}{S(m,L)}
\newcommand{\cvxx}{S(0,L-m)}

\newcommand{\cvxw}{S(0,L)}

\numberwithin{equation}{section}

\title{Analysis and Design of Optimization Algorithms via Integral Quadratic Constraints} 

\if\MODE2
\author{Laurent Lessard\and  Benjamin Recht\and Andrew Packard \thanks{L.~Lessard is with the Department of Electrical and Computer Engineering at the University of Wisconsin--Madison. \texttt{laurent.lessard@wisc.edu}. B.~Recht is with the Departments of Electrical Engineering, Computer Science, and Statistics at the University of California, Berkeley. \texttt{brecht@berkeley.edu}. A.~Packard is with the Department of Mechanical Engineering at the University of California, Berkeley. \texttt{apackard@berkeley.edu}.}}
\else
\author{Laurent Lessard \and  Benjamin Recht \and Andrew Packard }
\note{}
\fi

\begin{document}

\maketitle

%%%%%%%%%%%%%%%%%%%%%%%%%%%%%%%%%%%%%%%%%%%%%%%%%%%%%%%%%%%%%%%%%%%%%%%%%%%%%%%
% the footnotes

\begin{abstract}
This manuscript develops a new framework to analyze and design iterative optimization algorithms built on the notion of Integral Quadratic Constraints (IQC) from robust control theory.  IQCs provide sufficient conditions for the stability of complicated interconnected systems, and these conditions can be checked by semidefinite programming.  We discuss how to adapt IQC theory to study optimization algorithms, proving new inequalities about convex functions and providing a version of IQC theory adapted for use by optimization researchers.  Using these inequalities, we derive numerical upper bounds on convergence rates for the Gradient method, the Heavy-ball method, Nesterov's accelerated method, and related variants by solving small, simple semidefinite programming problems.  We also briefly show how these techniques can be used to search for optimization algorithms with desired performance characteristics, establishing a new methodology for algorithm design.
\end{abstract}

\if\MODE2
\fi
%%%%%%%%%%%%%%%%%%%%%%%%%%%%%%%%%%%%%%%%%%%%%%%%%%%%%%%%%%%%%%%%%%%%%%%%%%%%%%%
\section{Introduction}

Convex optimization algorithms provide a powerful toolkit for robust, efficient, large-scale optimization algorithms.  They provide not only effective tools for solving optimization problems, but are guaranteed to converge to accurate solutions in provided time budgets~\cite{NesterovBook,nesterov-nemirovskii}, are robust to errors and time delays~\cite{Bertsekas86,Nemirovski09}, and are amendable to declarative modeling that decouples the algorithm design from the problem formulation~\cite{tfocs,gb08,YALMIP}.  However, as we push up against the boundaries of the convex analysis framework, try to build more complicated models, and aim to deploy optimization systems in highly complex environments, the mathematical guarantees of convexity start to break. The standard proof techniques for analyzing convex optimization rely on deep insights by experts and are devised on an algorithm-by-algorithm basis. It is thus not clear how to extend the toolkit to more diverse scenarios where multiple objectives---such as robustness, accuracy, and speed---need to be delicately balanced.

This paper marks an attempt at providing a systematized approach to the design and analysis optimization algorithms using techniques from control theory.  Our strategy is to adapt the notion of an \emph{integral quadratic constraint} from robust control theory~\cite{megrantzer}.  These constraints link sequences of inputs and outputs of operators, and are ideally suited to proving algorithmic convergence.  We will see that for convex functions, we can derive these constraints using only the standard first-order characterization of convex functions, and that these inequalities will be sufficient to reduce the analysis of first-order methods to the solution of a very small semidefinite program. Our IQC framework puts the analysis of algorithms in a  unified proof framework, and enables new analyses of algorithms by minor perturbations of existing proofs. This new system aims to simplify and automate the analysis of optimization programs, and perhaps to open new directions for algorithm design.

Our methods are inspired by the recent work of Drori and Teboulle \cite{drori_teboulle}.  In their manuscript, the authors propose writing down the first-order convexity inequality for all steps of an algorithmic procedure.  They then derive a semidefinite program that analytically verifies very tight bounds for the convergence rate for the Gradient method, and numerically precise bounds for convergence of Nesterov's method and other first-order methods.  The main drawback of the Drori and Teboulle approach is that the size of the semidefinite program scales with the number of time steps desired.  Thus, it becomes computationally laborious to analyze algorithms that require more than a few hundred iterations.  

Integral quadratic constraints will allow us to circumvent this issue. A typical example of one of our semidefinite programs might have a $3\times 3$ positive semidefinite decision variable, 3 scalar variables, a $5\times 5$ semidefinite cone constraint, and 4 scalar constraints. Such a problem can be solved in less than 10 milliseconds on a laptop with standard solvers.

We are able to analyze a variety of methods in our framework.  We show that our framework recovers the standard rates of convergence for the Gradient method applied to strongly convex functions.  We show that we can numerically estimate the performance of Nesterov's method. Indeed, our analysis provides slightly sharper bounds than Nesterov's proof. We show how our system fails to certify the stability of the popular Heavy-ball method of Polyak for strongly convex functions whose condition ratio is larger than 18.  Based on this analysis, we are able to construct a one-dimensional strongly convex function whose condition ratio is 25 and prove analytically that the Heavy-ball method fails to find the global minimum of this function.  This suggests that our tools can also be used as a way to guide the construction of counterexamples.  

We show that our methods extend immediately to the projected and proximal variants of all the first order methods we analyze.  We also show how to extend our analysis to functions that are convex but not strongly convex, and provide bounds on convergence that are within a logarithmic factor of the best upper bounds.  We also demonstrate that our methods can bound convergence rates when the gradient is perturbed by relative deterministic noise.  We show how different parameter settings lead to very different degradations in performance bounds as the noise increases.

Finally, we turn to algorithm \emph{design}.  Since our semidefinite program takes as input the parameters of our iterative scheme, we can search over these parameters.  For simple two-step methods, our algorithms are parameterized by 3 parameters, and we show how we can derive first-order methods that achieve nearly the same rate of convergence as Nesterov's accelerated method but are  more robust to noise.

The manuscript is organized as follows. We begin with a discussion of discrete-time dynamical system and how common optimization algorithms can be viewed as feedback interconnections between a known \emph{linear} system with an uncertain \emph{nonlinear} component.  We then turn to show how quadratic Lyapunov functions can be used to certify rates of convergence for optimization problems and can be found by semidefinite programming.  This immediately leads to the notion of an integral quadratic constraint. Another contribution of this work is a new form of IQC analysis geared specifically toward rate-of-convergence conclusions, and accessible to optimization researchers.
We also discuss their history in robust control theory and how they can be derived.  With these basic IQCs in hand, we then turn to analyzing the Gradient method and Nesterov method, their projected and proximal variants, and their robustness to noise.  We discuss one possible brute-force technique for designing new algorithms, and how we can outperform existing methods.  Finally, we conclude with many directions for future work.

\subsection{Notation and conventions}\label{sec:notation}

\paragraph{Common matrices.} The $d\times d$ identity matrix and zero matrix are denoted $I_d$ and $0_d$, respectively. Subscripts are omitted when they are to be inferred by context.

\paragraph{Norms and sequences.} We define $\ltwoe^n$ to be the set of all one-sided sequences $x: \N \to \R^n$. We sometimes omit $n$ and simply write $\ltwoe$ when the superscript is clear from context. 
The notation $\|\cdot\|:\R^n\to \R$ denotes the standard 2-norm. The subset $\ltwo \subset \ltwoe$ consists of all square-summable sequences. In other words, $x\in\ltwo$ if and only if 
$
\sum_{k=0}^\infty \|x_k\|^2
$
is convergent.

\paragraph{Convex functions.} For a given $0<m<L$, we define $\cvx$ to be the set of functions $f:\R^d\to \R$ that are continuously differentiable, strongly convex with parameter $m$, and have Lipschitz gradients with parameter $L$. In other words, $f$ satisfies
\[
m \|x - y\|^2 \le (\grad f(x) - \grad f(y))^\tp (x-y) \le L \|x-y\|^2
\quad\text{for all }x,y\in\R^d
\]
We call $\kappa \defeq L/m$ the \emph{condition ratio} of $f\in\cvx$. We adopt this terminology to distinguish the condition ratio of a function from the related concept of \emph{condition number} of a matrix. The connection is that if $f$ is twice differentiable, we have the bound: $\cond(\hess f(x)) \le \kappa$ for all $x\in\R^d$, where $\cond(\cdot)$ is the condition number.

\paragraph{Kronecker product}
The Kronecker product of two matrices $A\in\R^{m\times n}$ and $B\in\R^{p\times q}$ is denoted $A\otimes B \in \R^{mp\times nq}$ and given by:
\[
A\otimes B = \bmat{ A_{11} B & \dots  & A_{1n} B \\
	                \vdots   & \ddots & \vdots   \\
	                A_{m1} B & \dots  & A_{mn} B }
\]
Two useful properties of the Kronecker product are that $(A\otimes B)^\tp = A^\tp \otimes B^\tp$ and that $(A\otimes B)(C\otimes D) = (AC)\otimes (BD)$ whenever the matrix dimensions are such that the products $AC$ and $BD$ make sense.

\section{Optimization algorithms as dynamical systems}\label{sec:opt-dyn}

A linear dynamical system is a set of recursive linear equations of the form
\begin{subequations}\label{sseq}
\begin{align}
	\xi_{k+1} &= A \xi_k + B u_k\\
	y_k &= C \xi_k + D u_k\,.
\end{align}
\end{subequations}
At each timestep $k=0,1,\dots$, $u_k \in \R^d$ is the \emph{input}, $y_k \in \R^d$ is the \emph{output}, and $\xi_k\in\R^m$ is the \emph{state}.
%When $u_k=0$, $\xi_k$ is sufficient information to predict all future outputs $y_k$.
We can write the dynamical system~\eqref{sseq} compactly by stacking the matrices into a block using the notation
\[
\stsp{A}{B}{C}{D}.
\]
We can connect this linear system in \emph{feedback} with a nonlinearity $\phi$ by defining the rule
\begin{subequations}\label{eq:lure}
\begin{align}
	\xi_{k+1}&= A \xi_k + B u_k\\
	y_k &= C \xi_k + D u_k\\
	u_k &= \phi(y_k)\,.
\end{align}
\end{subequations}
In this case, the output is transformed by the map $\phi:\R^d\to\R^d$ and is then used as the input to the linear system.

In this paper, we will be interested in the case when the interconnected nonlinearity has the form $\phi(y) = \grad f(y)$ where $f\in \cvx$.  In particular, we will consider algorithms designed to solve the optimization problem
\begin{equation}\label{eq:main-opt}
\minimize_{x\in\R^n} f(x)
\end{equation}
as dynamical systems and see how this new viewpoint can give us insights into convergence analysis.  Section~\ref{sec:prox-point} considers variants of~\eqref{eq:main-opt} where the decision variable $x$ is constrained or $f$ is non-smooth.

Standard first order methods such as the Gradient method, Heavy-ball method, and Nesterov's accelerated method, can all be cast in the form~\eqref{eq:lure}.  In all cases, the nonlinearity is the mapping $\phi(y) = \grad f(y)$.  The state transition matrices $A$, $B$, $C$, $D$ differ for each algorithm.  The Gradient method can be expressed as
\begin{equation}\label{aa}
\left[\begin{array}{c|c}
    A & B \\ \hlinet
    C & D
\end{array}\right] = \left[\begin{array}{c|c}
    I_d & -\alpha I_d \\ \hlinet
    I_d & 0_d
\end{array}\right]\,.
\end{equation}
To verify this, substitute~\eqref{aa} into \eqref{eq:lure} and obtain
\[
\begin{aligned}
	\xi_{k+1}&= \xi_k - \alpha u_k\\
y_k &= \xi_k \\
u_k &= \grad f(y_k)
\end{aligned}
\]
Eliminating $y_k$ and $u_k$ and renaming $\xi$ to $x$ yields
\[
	x_{k+1} = x_k - \alpha \grad f(x_k)
\]
which is the familiar Gradient method with constant stepsize.
Nesterov's accelerated method for strongly convex functions is given by the dynamical system
\begin{equation}\label{bb}
\left[\begin{array}{c|c}
    A & B \\ \hlinet
    C & D
\end{array}\right] =\left[\begin{array}{cc|c}
    (1+\beta)I_d & -\beta I_d  & -\alpha I_d\\
    I_d & 0_d & 0_d\\\hlinet
    (1+\beta)I_d & -\beta I_d & 0_d
\end{array}\right]
\end{equation}
Verifying that \eqref{bb} is equivalent to Nesterov's method takes only slightly more effort than it did for the Gradient method.  Substituting~\eqref{bb} into~\eqref{eq:lure} now yields
\begin{subequations}
\begin{align}
\xi^{(1)}_{k+1} &=  (1+\beta) \xi^{(1)}_k -\beta \xi^{(2)}_k -\alpha u_k \label{eq:nest-delay1}\\
\label{eq:nest-delay} \xi^{(2)}_{k+1} &= \xi^{(1)}_k\\ \label{eq:nest-delay2}
y_k &= (1+\beta) \xi^{(1)}_k - \beta \xi^{(2)}_k\\ \label{eq:nest-delay3}
u_k &= \grad f (y_k)
\end{align}
\end{subequations}
Note that~\eqref{eq:nest-delay} asserts that the partial state $\xi^{(2)}$ is a delayed version of the state $\xi^{(1)}$. Substituting \eqref{eq:nest-delay} into \eqref{eq:nest-delay1} gives the simplified system
\begin{align*}
\xi^{(1)}_{k+1} &=  (1+\beta) \xi^{(1)}_k -\beta \xi^{(1)}_{k-1} -\alpha u_k\\
y_k &= (1+\beta) \xi^{(1)}_k - \beta \xi^{(1)}_{k-1}\\
u_k &= \grad f (y_k) 
\end{align*}
Eliminating $u_k$ and renaming $\xi^{(1)}$ to $x$ yields the common form of Nesterov's method
\begin{align*}
x_{k+1} &=  y_k -\alpha \grad f (y_k)\\
y_k &= (1+\beta) x_k - \beta x_{k-1}\,.
\end{align*}
Note that other variants of this algorithm exist for which the $\alpha$ and $\beta$ parameters are updated at each iteration. In this paper, we restrict our analysis to the constant-parameter version above. The Heavy-ball method is given by
\begin{equation}\label{cc}
\left[\begin{array}{c|c}
    A & B \\ \hlinet
    C & D
\end{array}\right] =\left[\begin{array}{cc|c}
    (1+\beta)I_d & -\beta I_d  & -\alpha I_d\\
    I_d & 0_d & 0_d\\ \hlinet
    I_d & 0_d & 0_d
\end{array}\right]
\end{equation}
One can check by similar analysis that~\eqref{cc} is equivalent to the update rule
\[
	x_{k+1} = x_{k} - \alpha \grad f(x_k) + \beta(x_k - x_{k-1})\,.
\]

\subsection{Proving algorithm convergence}\label{sec:howdoweprove}
Convergence analysis of convex optimization algorithms typically follows a two step procedure.  First one must show that the algorithm has a fixed point that solves the optimization problem in question.  Then, one must verify that from a reasonable starting point, the algorithm converges to this optimal solution at a specified rate.

In dynamical systems, such proofs are called stability analysis. By writing common first order methods as dynamical systems, we can unify their stability analysis.  For a general problem with minimum occurring at $y_\star$, a necessary condition for optimality is that $u_\star = \grad f(y_\star) = 0$. Substituting into~\eqref{sseq}, the fixed point satisfies
\[
	y_\star = C \xi_\star \quad\text{and}\quad
	\xi_\star = A \xi_\star
\]
In particular, $A$ must have an eigenvalue of $1$. If the blocks of $A$ are diagonal as in the Gradient, Heavy-ball, or Nesterov methods shown above, then the eigenvalue of $1$ will have a geometric multiplicity of at least $d$.

Proving that all paths lead to the optimal solution requires more effort and constitutes the bulk of what is studied herein.  Before we proceed for general convex $f$, it is instructive to study what happens for quadratic $f$.

\subsection{Quadratic problems}\label{sec:quadratic_case}
Suppose $f$ is a convex, quadratic function
$
	f(y) = \tfrac{1}{2} y^\tp Q y -p^\tp y +r
$,
where $m I_d \preceq Q \preceq L I_d$ in the positive definite ordering.  The gradient of $f$ is simply
$
	\grad f(y) = Qy - p
$	
and the optimal solution is $y_\star = Q^{-1}p$.  

What happens when we run a first order method on a quadratic problem?  Assume throughout this section that $D=0$.  Substituting the equation for $y_\star$ and $\grad f(y)$ back into~\eqref{eq:lure}, we obtain the system of equations:
\begin{align*}
\xi_{k+1} &= A \xi_k + B u_k \\
y_k &= C \xi_k \\
u_k &= \grad f(y_k) = Q y_k - p = Q(y_k - y_\star)
\end{align*}
Now make use of the fixed-point equations $y_\star = C \xi_\star$ and $\xi_\star = A \xi_\star$ and we obtain $u_k = QC(\xi_k - \xi_\star)$. Eliminating $y_k$ and $u_k$ from the above equations, we obtain
\begin{equation}\label{xx}
	\xi_{k+1}-\xi_\star= (A + B QC) (\xi_k-\xi_\star)
\end{equation}
Let $T\defeq A+BQC$ denote the closed-loop state transition matrix.  A necessary and sufficient condition for $\xi_k$ to converge to $\xi_\star$ is that the \emph{spectral radius} of $T$ is strictly less than $1$.  Recall that the spectral radius of a matrix $M$ is defined as the largest magnitude of the eigenvalues of $M$.  We denote the spectral radius by $\rho(M)$.
It is a fact that
\[
	\rho(M) \le \|M^k\|^{1/k}
	\qquad\text{for all $k$ and}\qquad
	\rho(M) = \lim_{k\to\infty} \|M^k\|^{1/k}
\]
where $\|\cdot\|$ is the induced $2$-norm. Therefore, for any $\epsilon > 0$, we have for all $k$ sufficiently large that $\rho(T)^k \le \|T^k\| \le (\rho(T)+\epsilon)^k$. Hence, we can bound the convergence rate:
\begin{align*}
\|\xi_{k}-\xi_\star\| = 
\|T^k(\xi_{0}-\xi_\star)\| \le
\|T^k\| \|\xi_{0}-\xi_\star\| \le
(\rho(T)+\epsilon)^k \|\xi_{0}-\xi_\star\|\,.
\end{align*}
So the spectral radius also determines the rate of convergence of the algorithm.
With only bounds on the eigenvalues of $Q$, we can provide conditions under which the algorithms above converge for quadratic $f$.

\begin{prop} \label{prop:quadratic-rates} The following table gives worst-case rates for different algorithms and parameter choices when applied to a class of \textbf{convex quadratic functions}. We assume here that $f:\R^d\to \R$ where $f(x)=\tfrac{1}{2} x^\tp Q x - p^\tp x + r$ and $Q$ is any matrix that satisfies $m I_d \preceq Q \preceq L I_d$.  We also define $\kappa \defeq L/m$. 
	
\begin{center}
	\smallskip
\begin{tabular}{|l|l|l|l|}
	\hlinet
	Method & Parameter choice & Rate bound & Comment \\ \hlinet
	Gradient & $\alpha = \frac{1}{L}$ & $\rho = 1-\frac{1}{\kappa}$ & popular choice \\
	Nesterov & $\alpha = \frac{1}{L},\, \beta = \frac{\sqrt{\kappa}-1}{\sqrt{\kappa}+1}$ & $\rho = 1-\frac{1}{\sqrt{\kappa}}$ & standard choice \\ \hlinet
	Gradient & $\alpha = \frac{2}{L+m}$ & $\rho = \frac{\kappa-1}{\kappa+1}$ & optimal tuning \\
	Nesterov & $\alpha = \frac{4}{3L+m},\, \beta = \frac{\sqrt{3\kappa+1}-2}{\sqrt{3\kappa+1}+2}$ & $\rho = 1-\frac{2}{\sqrt{3\kappa+1}}$ & optimal tuning \\
	Heavy-ball & $\alpha = \frac{4}{(\sqrt L + \sqrt m)^2},\, \beta = \bigl(\frac{\sqrt{\kappa}-1}{\sqrt{\kappa}+1}\bigr)^2$ & $\rho = \frac{\sqrt{\kappa}-1}{\sqrt{\kappa}+1}$ & optimal tuning \\\hline
\end{tabular}
\smallskip
\end{center}
\end{prop}
All of these results are proven by elementary linear algebra and the bounds are tight. In other words, there exists a quadratic function that achieves the worst-case $\rho$. See Appendix~\ref{A:prop1proof} for more detail.

Unfortunately, the proof technique used in Proposition~\ref{prop:quadratic-rates} does not extend to the case where $f$ is a more general strongly convex function.   However, a different characterization of stability does generalize and will be described in Section~\ref{sec:iqc}. It turns out that for linear systems, stability is equivalent to the feasibility of a particular semidefinite program.  We will see in the sequel that similar semidefinite programs can be used to certify stability of nonlinear systems.

\begin{prop}\label{prop:lyap}
Suppose $T\in\R^{d\times d}$. Then $\rho(T) < \rho$ if and only if there exists a $P\succ 0$ satisfying
$T^\tp PT - \rho^2 P \prec 0$.
\end{prop}

The proof of Proposition~\ref{prop:lyap} is elementary so we omit it.
The use of Linear Matrix Inequalities (LMI) to characterize stability of a linear time-invariant system dates back to Lyapunov~\cite{lyapunov}, and we give a more detailed account of this history in Section~\ref{sec:history_lyap}.
%
%\begin{proof}
%If $\rho(T) < \rho$, then the matrix
%\[
%	P \defeq \sum_{k=0}^\infty  \rho^{-2k}(T^k)^\tp (T^k) 
%\]
%is well defined, positive definite because the first term in the sum is a multiple of the identity, and satisfies $T^\tp PT - \rho^2 P =-\rho^2 I_d \prec 0$.  Conversely, assume the LMI has a solution $P\succ 0$ and let $\lambda$ be an eigenvalue of $T$ with corresponding eigenvector $q$.  Then
%\[
%	0>q^* T^\tp P T q - \rho^2 q^* P q = (|\lambda|^2-\rho^2) q^* P q 
%\]
%But since $q^* P q > 0$, we must have that $|\lambda|<\rho$.
%\end{proof}
%
Now suppose we are studying a dynamical system of the form $\xi_{k+1}-\xi_\star = T (\xi_k-\xi_\star)$ as in~\eqref{xx}.  Then, if there exists a $P\succ 0$ satisfying $T^\tp PT - \rho^2 P \prec 0$,
\begin{equation}\label{yy}
	(\xi_{k+1}-\xi_\star)^\tp P (\xi_{k+1}-\xi_\star) < \rho^2 (\xi_{k}-\xi_\star)^\tp P (\xi_{k}-\xi_\star)
\end{equation}
along all trajectories. If $\rho < 1$, then the sequence $\{\xi_{k}\}_{k\ge 0}$ converges linearly to $\xi_\star$. Iterating~\eqref{yy} down to $k=0$, we see that
\begin{equation}\label{yyy}
	(\xi_{k}-\xi_\star)^\tp P (\xi_{k}-\xi_\star) < \rho^{2k} (\xi_{0}-\xi_\star)^\tp P (\xi_{0}-\xi_\star)
\end{equation}
which implies that
\begin{equation}\label{z}
	\|\xi_{k}-\xi_\star\| < \sqrt{\cond(P)}\, \rho^k \|\xi_{0}-\xi_\star\|
\end{equation}
where $\cond(P)$ is the condition number of $P$. In what follows, we will generalize this semidefinite programming approach to yield feasibility problems that are sufficient to characterize when the closed loop system~\eqref{eq:lure} converges and which provide bounds on the distance to optimality as well.
The function
\begin{equation}\label{ww}
	V(\xi) = (\xi-\xi_\star)^\tp P (\xi-\xi_\star)
\end{equation}
is called a \emph{Lyapunov function} for the dynamical system.  This function strictly decreases over all trajectories and hence certifies that the algorithm is \emph{stable}, i.e., converges to nominal values.  The conventional method for proving stability of an electromechanical system is to show that some notion of \emph{total energy} always decreases over time.  Lyapunov functions provide a convenient mathematical formulation of this notion of total energy.

The question for the remainder of the paper is how can we search for Lyapunov-like functions that guarantee algorithmic convergence when $f$ is not quadratic.

\section{Proving convergence using integral quadratic constraints}\label{sec:iqc}

When the function being minimized is quadratic as explored in Section~\ref{sec:quadratic_case}, its gradient is affine and the interconnected dynamical system is a simple linear difference equation
whose stability and convergence rate is analyzed solely in terms of
eigenvalues of the closed-loop system.  When the cost function is not quadratic, the gradient update is not an affine function and hence a different analysis technique is required.

A popular technique in the control theory literature is to use~\emph{integral quadratic constraints} (IQCs) to capture features of the behavior of partially-known components. The term IQC was introduced in the seminal paper by Megretski and Rantzer~\cite{megrantzer}. In that work, the authors analyzed continuous time dynamical systems and the constraints involved integrals of quadratic functions, hence the name IQC. 

In the development that follows, we repurpose the classical IQC theory for use in algorithm analysis. This requires using discrete time dynamical systems so our constraints will involve sums of quadratics rather than integrals. We also adapt the theory in a way that allows us to certify a specific convergence rate in addition to stability.

\subsection{An introduction to IQCs}\label{subsec:iqc_defn}

IQCs provide a convenient framework for analyzing interconnected dynamical systems that contain components that are noisy, uncertain, or otherwise difficult to model. The idea is to replace this troublesome component by a quadratic constraint on its inputs and outputs that is known to be satisfied by all possible instances of the component. If we can certify that the newly constrained system performs as desired, then the original system must do so as well.

Suppose $\phi:\ltwoe\to\ltwoe$ is the troublesome function we wish to analyze. The equation $u=\phi(y)$ can be represented using a block diagram, as in Figure~\ref{fig:phi}.
\begin{figure}[ht]
\centering
\includegraphics{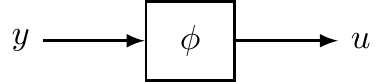}
\caption{Block-diagram representation of the map $\phi$.\label{fig:phi}}
\end{figure}

Although we do not know $\phi$ exactly, we assume that we have some knowledge of the constraints it imposes on the pair $(y,u)$. For example, suppose it is known that $\phi$ satisfies the following properties: 
\begin{enumerate}[(i)]
\item $\phi$ is static and memoryless: $\phi(y_0,y_1,\dots) = (g(y_0),g(y_1),\dots)$
for some $g:\R^d\to\R^d$.
\item $g$ is $L$-Lipschitz: $\|g(y_1)-g(y_2)\| \le L \|y_1-y_2\|$ for all $y_1,y_2\in\R^d$.
\end{enumerate}
Now suppose that $y = (y_0,y_1,\dots)$ is an arbitrary sequence of vectors in $\R^d$, and $u = \phi(y)$ is the output of the unknown function applied to $y$. Property~(ii) implies that $\|u_k-u_\star\| \le L \|y_k-y_\star\|$ for all $k$, where $(y_\star,u_\star)$ is any pair of vectors satisfying $u_\star = g(y_\star)$ that will serve as a reference point. In matrix form, this is
\begin{equation}\label{eq:norm_bound2}
\bmat{y_k -y_\star \\ u_k-u_\star}^\tp \bmat{ L^2 I_d & 0_d \\ 0_d & -I_d } \bmat{y_k-y_\star \\ u_k-u_\star} \ge 0
\qquad\text{for }k=0,1,\dots
\end{equation}%
\paragraph{Core idea behind IQC.} Instead of analyzing a system that contains $\phi$, we analyze the system where $\phi$ is removed, but we enforce the constraints~\eqref{eq:norm_bound2} on the signals $(y,u)$. Since~\eqref{eq:norm_bound2} is true for all admissible choices of~$\phi$, then any properties we can prove for the constrained system must hold for the original system as well.

Note that~\eqref{eq:norm_bound2} is rather special in that the quadratic coupling of $(y,u)$ is pointwise; it only manifests itself as separate quadratic constraints on each $(y_k,u_k)$. It is possible to specify more general quadratic constraints that couple different $k$ values, and the key insight above still holds. To do this, introduce auxiliary sequences $\zeta,z\in\ltwoe$ together with a map $\Psi$ characterized by the matrices $(A_\Psi,B_\Psi^y,B_\Psi^u,C_\Psi,D_\Psi^y,D_\Psi^u)$ and the recursion%
\begin{subequations}\label{psi}
\begin{align}
\zeta_0 &= \zeta_\star \label{zeta_init}\\
\zeta_{k+1} &= A_\Psi \zeta_k + B_\Psi^y  y_k +  B_\Psi^u u_k \\
z_{k} &= C_\Psi \zeta_k + D_\Psi^y y_k  +  D_\Psi^u u_k
\end{align}
\end{subequations}
where we will define the initial condition $\zeta_\star$ shortly. The equations~\eqref{psi} define an affine map $z = \Psi(y,u)$.  Assuming a reference point $(y_\star,u_\star)$ as before, we can define the associated reference $(\zeta_\star,z_\star)$ that is a fixed point of~\eqref{psi}. In other words,
\begin{subequations}\label{psi0}
\begin{align}
\zeta_\star &= A_\Psi \zeta_\star + B_\Psi^y  y_\star +  B_\Psi^u u_\star \label{psi01}\\
z_\star &= C_\Psi \zeta_\star + D_\Psi^y y_\star  +  D_\Psi^u u_\star\label{psi02}
\end{align}
\end{subequations}
We will require that $\rho(A_\Psi) < 1$, which ensures that~\eqref{psi0} has a unique solution~$(\zeta_\star,z_\star)$ for any choice of $(y_\star,u_\star)$. Note that the reference points are defined in such a way that if we use $y = (y_\star,y_\star,\dots)$ and $u = (u_\star,u_\star,\dots)$ in~\eqref{psi}, we will obtain $\zeta = (\zeta_\star,\zeta_\star,\dots)$ and $z = (z_\star,z_\star,\dots)$.

We then consider the quadratic forms $(z_k-z_\star)^\tp M (z_k-z_\star)$ for a given symmetric matrix $M$ (typically indefinite). Note that each such quadratic form is a function of $(y_0,\dots,y_k,u_0,\dots,u_k)$ that is determined by our choice of $(\Psi,M,y_\star,u_\star)$. In our previous example~\eqref{eq:norm_bound2}, $\Psi$ has no dynamics and the corresponding $\Psi$ and $M$ are
\begin{equation}\label{harff}
\Psi = \left[\begin{array}{c|cc}
A_\Psi & B_\Psi^y & B_\Psi^u \\ \hlinet
C_\Psi & D_\Psi^y & D_\Psi^u
\end{array}\right] =
\left[\begin{array}{c|cc}
0_d & 0_d & 0_d \\ \hlinet
0_d & I_d & 0_d \\
0_d & 0_d & I_d 
\end{array}\right]
\qquad
M = \bmat{L^2 I_d & 0_d \\ 0_d & -I_d}
\end{equation}
In other words, if we use the definitions~\eqref{harff}, then $(z_k-z_\star)^\tp M (z_k-z_\star) \ge 0$ is the same as~\eqref{eq:norm_bound2}. In general, these sorts of quadratic constraints are called IQCs. We consider four different types of IQCs, which we now define.

\begin{defn}\label{def:iqc}
Suppose $\phi:\ltwoe^d\to\ltwoe^d$ is an unknown map and $\Psi:\ltwoe^d\times\ltwoe^d\to\ltwoe^m$ is a given linear map of the form~\eqref{psi} with $\rho(A_\Psi)<1$. Suppose $(y_\star,u_\star) \in \R^{2d}$ is a given reference point and let $(\zeta_\star,z_\star)$ be the unique solution of~\eqref{psi0}. Suppose $y\in\ltwo^d$ is an arbitrary square-summable sequence. Let $u = \phi(y)$ and let $z = \Psi(y,u)$ as in~\eqref{psi}. We say that $\phi$ satisfies the
\begin{enumerate}
\item \textbf{Pointwise IQC} defined by $(\Psi,M,y_\star,u_\star)$ if for all $y\in\ltwoe^d$ and $k\ge 0$,
\[
(z_k-z_\star)^\tp M (z_k-z_\star) \ge 0
\]
\item \textbf{Hard IQC} defined by $(\Psi,M,y_\star,u_\star)$ if for all $y\in\ltwoe^d$ and $k\ge 0$,
\[
\sum_{t=0}^k (z_t-z_\star)^\tp M (z_t-z_\star) \ge 0
\]
\item $\rho$-\textbf{Hard IQC} defined by $(\Psi,M,\rho,y_\star,u_\star)$ if for all $y\in\ltwoe^d$ and $k\ge 0$,
\[
\sum_{t=0}^k \rho^{-2t} (z_t-z_\star)^\tp M (z_t-z_\star) \ge 0
\]
\item $\textbf{Soft IQC}$ defined by $(\Psi,M,y_\star,u_\star)$ if for all $y\in\ltwo^d$,
\[
\sum_{t=0}^\infty (z_t-z_\star)^\tp M (z_t-z_\star) \ge 0
\qquad\text{(and the sum is convergent)}
\]
\end{enumerate}
\end{defn}
Note that the example~\eqref{eq:norm_bound2} is a pointwise IQC. Examples of the other types of IQCs will be described in Section~\ref{subsec:iqc_library}. Note that the sets of maps satisfying the various IQCs defined above are nested as follows:
\[
\{\text{all pointwise IQCs}\} \subset
\{\text{all $\rho$-hard IQCs, }\rho < 1\} \subset
\{\text{all hard IQCs}\} \subset
\{\text{all soft IQCs}\}
\]
For example, if $\phi$ satisfies a pointwise IQC defined by $(\Psi,M,y_\star,u_\star)$ then it must also satisfy the hard IQC defined by the same $(\Psi,M,y_\star,u_\star)$.
The notions of \emph{hard IQC} and the more general \emph{soft IQC} (sometimes simply called \emph{IQC}) were introduced in~\cite{megrantzer} and their relationship is discussed in~\cite{seiler13TAC}. These concepts are useful in proving that a dynamic system is stable, but do not directly allow for the derivation of useful bounds on convergence rates. The definitions of \emph{pointwise} and $\rho$-\emph{hard} IQCs are new, and were created for the purpose of better characterizing convergence rates, as we will see in Section~\ref{subsec:stab_perf}.

Finally, note that $y_\star$ and $u_\star$ are nominal inputs and outputs for the unknown $\phi$, and they can be tuned to certify different fixed points of the interconnected system. We will see in Section~\ref{subsec:stab_perf} that certifying a particular convergence rate to some fixed point does not require prior knowledge of fixed point; only knowledge that the fixed point exists.

%%%%%%%%%%%%%%%%%%%%%%%%%%%%%%%%%%%%%%%%%%%%%%%%%%%%%%%%%%%%%%%%%%%%%%%%%%%%%%%
\subsection{Stability and performance results}\label{subsec:stab_perf}

In this section, we show how IQCs can be used to prove that iterative algorithms converge and to bound the rate of convergence. In both cases, the certification requires solving a tractable convex program. We note that the original work on IQCs~\cite{megrantzer} only proved stability (boundedness). Some other works have addressed exponential stability~\cite{JonExp,RanMegExp,MegRanPartII}, but the emphasis of these works is on proving the \emph{existence} of an exponential decay rate, and so the rates constructed are very conservative. We require rates that are less conservative, and this is reflected in the inclusion of $\rho$ in the LMI of our main result, Theorem~\ref{thm:main}.

We will now combine the dynamical system framework of Section~\ref{sec:opt-dyn} and the IQC theory of Section~\ref{subsec:iqc_defn}. Suppose $G:\ltwoe^d\to\ltwoe^d$ is an affine map $u\mapsto y$ described by the recursion
\begin{subequations}\label{eqG}
\begin{align}
\xi_{k+1} &= A \xi_k + B u_k \\
y_k &= C \xi_k
\end{align}
\end{subequations}
where $(A,B,C)$ are matrices of appropriate dimensions. The map is affine rather than linear because of the initial condition $\xi_0$. As in Section~\ref{sec:opt-dyn}, $G$ is the iterative algorithm we wish to analyze, and using the general formalism of Section~\ref{subsec:iqc_defn}, $\phi$ is the nonlinear map $(y_0,y_1,\dots)\mapsto(u_0,u_1,\dots)$ that characterizes the feedback. Of course, this framework subsumes the special case of interest in which $u_k = \grad f(y_k)$ for each $k$. We assume that $\phi$ satisfies an IQC, and this IQC is characterized by a map $\Psi$ and matrix $M$. We can interpret $z = \Psi(y,u)$ as a filtered version of the signals $u$ and $y$.  These equations can be represented using a block-diagram as in Figure~\ref{fig:Gphi_blk2a}.

\begin{figure}[ht]
\centering
\begin{subfigure}[t]{0.47\linewidth}
	\centering
	\includegraphics{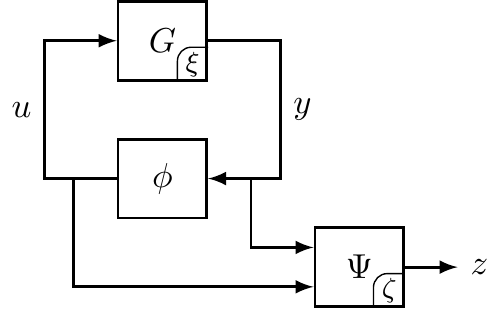}
	\caption{The auxiliary system $\Psi$ produces $z$, which is a filtered version of the signals $y$ and $u$.}\label{fig:Gphi_blk2a}
\end{subfigure}
\hfill
\begin{subfigure}[t]{0.47\linewidth}
	\centering
	\includegraphics{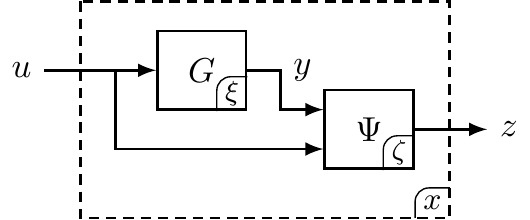}
	\caption{The nonlinearity $\phi$ is replaced by a constraint on $z$, so we may remove $\phi$ entirely.}\label{fig:Gphi_blk2b}
\end{subfigure}
\caption{Feedback interconnection between a system $G$ and a nonlinearity $\phi$. An IQC is a constraint on $(y,u)$ satisfied by $\phi$. We only analyze the constrained system and so we may remove the $\phi$ block entirely.}\label{fig:Gphi_blk2}
\end{figure}

Consider the dynamics of $G$ and $\Psi$ from~\eqref{eqG} and~\eqref{psi}, respectively. Upon eliminating~$y$, the recursions may be combined to obtain
\begin{subequations}\label{eq:comboss}
\begin{align}
\bmat{\xi_{k+1} \\ \zeta_{k+1}} &= 
\bmat{A & 0 \\ B_\Psi^y C & A_\Psi} \bmat{\xi_{k} \\ \zeta_{k}} +
\bmat{B \\ B_\Psi^u} u_k \\
z_k &= \bmat{D_\Psi^y C & C_\Psi} \bmat{\xi_{k} \\ \zeta_{k}} + D_\Psi^u u_k 
\end{align}
\end{subequations}
More succinctly, ~\eqref{eq:comboss} can be written as
\begin{equation}\label{eq:comboss2}
\begin{aligned}
x_{k+1}&=  \hat A x_k + \hat B u_k \\
z_k &=  \hat C x_k + \hat D u_k 
\end{aligned}
\qquad
\text{where we defined }x_k \defeq \bmat{ \xi_k \\ \zeta_k }
\end{equation}
The dynamical system~\eqref{eq:comboss2} is represented in Figure~\ref{fig:Gphi_blk2b} by the dashed box. Our main result is as follows.

\begin{thm}[Main result]\label{thm:main}
Consider the block interconnection of Figure~\ref{fig:Gphi_blk2a}.
Suppose $G$ is given by~\eqref{eqG} and $\Psi$ is given by~\eqref{psi}. 
Define $(\hat A,\hat B,\hat C,\hat D)$ as in \eqref{eq:comboss}--\eqref{eq:comboss2}.
Suppose $(\xi_\star,\zeta_\star,y_\star,u_\star,z_\star)$ is a fixed point of \eqref{eqG} and \eqref{psi}. In other words,
\begin{subequations}\label{koala}
\begin{align}
\xi_\star &= A \xi_\star + B u_\star \\
y_\star &= C\xi_\star \\
\zeta_\star &= A_\Psi \zeta_\star + B_\Psi^y  y_\star +  B_\Psi^u u_\star \\
z_\star &= C_\Psi \zeta_\star + D_\Psi^y y_\star  +  D_\Psi^u u_\star
\end{align}
\end{subequations}
Suppose $\phi$ satisfies the $\rho$-\textbf{hard IQC} defined by $(\Psi,M,\rho,y_\star,u_\star)$ where $0\le \rho \le 1$.  
Consider the following LMI.
\begin{equation}\label{eq:expLMI}
\bmat{ \hat A^\tp P \hat A - \rho^2 P & \hat A^\tp P \hat B \\
\hat B^\tp P \hat A & \hat B^\tp P \hat B} + \lambda
 \bmat{\hat C & \hat D}^\tp M \bmat{\hat C & \hat D}
\preceq 0
\end{equation}
If~\eqref{eq:expLMI} is feasible for some $P\succ 0$ and $\lambda \ge 0$, then for any $\xi_0$, we have \[
\|\xi_{k}-\xi_\star\| \le \sqrt{\cond(P)}\,\rho^{k}\, \|\xi_0-\xi_\star\|
\quad\text{for all }k
\]
where $\cond(P)$ is the condition number of $P$.
\end{thm}
\begin{proof}
Let $x,u,z\in\ltwoe$ be a set of sequences that satisfies~\eqref{eq:comboss2}. Suppose $(P,\lambda)$ is a solution of~\eqref{eq:expLMI}. Multiply~\eqref{eq:expLMI} on the left and right by $\bmat{(x_k-x_\star)^\tp & (u_k-u_\star)^\tp}$ and its transpose, respectively. Making use of~\eqref{eq:comboss2}--\eqref{koala}, we obtain
\begin{equation}\label{eq:simplerlmi}
(x_{k+1}-x_\star)^\tp P (x_{k+1}-x_\star) - \rho^2\, (x_k-x_\star)^\tp P (x_k-x_\star) + \lambda\, (z_k-z_\star)^\tp M (z_k-z_\star) \le 0
\end{equation}
Multiply~\eqref{eq:simplerlmi} by $\rho^{-2k}$ for each $k$ and sum over $k$. The first two terms yield a telescoping sum and we obtain
\begin{multline*}
\rho^{-2k+2}(x_{k}-x_\star)^\tp P (x_{k}-x_\star) - \rho^{2} (x_0-x_\star)^\tp P (x_0-x_\star) \\
+ \lambda\sum_{t=0}^{k-1} \rho^{-2t}(z_t-z_\star)^\tp M (z_t-z_\star) \le 0
\end{multline*}
Because $\phi$ satisfies the  $\rho$-hard IQC defined by $(\Psi,M,\rho,y_\star,u_\star)$, the summation part of the inequality is nonnegative for all $k$. Therefore,
\[
(x_{k}-x_\star)^\tp P (x_{k}-x_\star) \le \rho^{2k} (x_0-x_\star)^\tp P (x_0-x_\star)
\]
for all $k$ and consequently $\|x_{k}-x_\star\| \le \sqrt{\cond(P)}\,\rho^{k}\, \|x_0-x_\star\|$. Recall from~\eqref{eq:comboss2} that $x_k = (\xi_k,\zeta_k)$ and from~\eqref{zeta_init} that $\zeta_0 = \zeta_\star$. Therefore,
\begin{align*}
\|\xi_k-\xi_\star\|^2 &\le \|x_k-x_\star\|^2 \\
&\le \cond(P)\rho^{2k} \|x_0-x_\star\|^2 \\
&= \cond(P)\rho^{2k} \bl(\|\xi_0-\xi_\star\|^2 + \|\zeta_0-\zeta_\star\|^2\br)\\
&= \cond(P)\rho^{2k} \|\xi_0-\xi_\star\|^2
\end{align*}
and this completes the proof.
\end{proof}

\noindent We now make several comments regarding Theorem~\ref{thm:main}.

\paragraph{Pointwise and hard IQCs} Theorem~\ref{thm:main} can easily be adapted to other types of IQCs.
\begin{enumerate}[1.]
\item If the pointwise IQC defined by some $(\Psi,M,y_\star,u_\star)$ is satisfied, then so is the $\rho$-hard IQC defined by $(\Psi,M,\rho,y_\star,u_\star)$ for any $\rho$. Therefore, we may apply Theorem~\ref{thm:main} directly and ignore the $\rho$-hardness constraint. The smallest $\rho$ that makes~\eqref{eq:expLMI} feasible will correspond to the best exponential rate we can guarantee.

\item Hard IQCs are a special case of $\rho$-hard IQCs with $\rho = 1$. Therefore, if the LMI~\eqref{eq:expLMI} is feasible, Theorem~\ref{thm:main} guarantees that $\|\xi_{k}-\xi_\star\| \le \sqrt{\cond(P)} \|\xi_0-\xi_\star\|$. In other words, the iterates are bounded (but not necessarily convergent).

\item If a $\rho_1$-hard IQC is satisfied, then so is the $\rho$-hard IQC for any $\rho \ge \rho_1$. Also, if~\eqref{eq:expLMI} is feasible for some $\rho_2$, it will also be feasible for any $\rho \ge \rho_2$. Therefore, if we use a $\rho_1$-hard IQC and~\eqref{eq:expLMI} is feasible for $\rho_2$, then the smallest exponential rate we can guarantee is $\rho = \max(\rho_1,\rho_2)$.
\end{enumerate}

\paragraph{Multiple IQCs} Theorem~\ref{thm:main} can also be generalized to the case where $\phi$ satisfies multiple IQCs. Suppose $\phi$ satisfies the $\rho$-hard IQCs defined by $(\Psi_i,M_i,\rho,y_\star^{(i)},u_\star^{(i)})$ for $i=1,\dots,r$. Simply redefine the matrices $(\hat A,\hat B,\hat C,\hat D)$ in a manner analogous to~\eqref{eq:comboss2}, but where the output is now $(z_k^{(1)},\dots,z_k^{(r)})$. Instead of~\eqref{eq:expLMI}, use
\begin{equation}\label{eq:mainLMI2}
\bmat{ \hat A^\tp P \hat A - \rho^2 P & \hat A^\tp P \hat B \\
\hat B^\tp P \hat A & \hat B^\tp P \hat B} +
 \bmat{\hat C & \hat D}^\tp \bmat{\lambda_1 M_1 & &\\& \ddots & \\ & & \lambda_r M_r} \bmat{\hat C & \hat D}
\preceq 0
\end{equation}
where $\lambda_1,\dots,\lambda_r\ge 0$.
Thus, when~\eqref{eq:mainLMI2} is multiplied out as in~\eqref{eq:simplerlmi}, we now obtain
\begin{multline*}
(x_{k+1}-x_\star)^\tp P (x_{k+1}-x_\star) - \rho^2\, (x_k-x_\star)^\tp P (x_k-x_\star) \\
+ \sum_{i=1}^r\lambda_i\, (z^{(i)}_k-z_\star^{(i)})^\tp M_i (z^{(i)}_k-z_\star^{(i)}) \le 0
\end{multline*}
and the rest of the proof proceeds as in Theorem~\ref{thm:main}.

%%%%%%%%%%%%%%%%%%%%%%%%%%%%%%%%%%%%%%%%%%%%%%%%%%%%%%%%%%%%%%%%%%%%%%%%%%%%%%%%
\paragraph{Remark on Lyapunov functions}

In the quadratic case treated in Section~\ref{sec:quadratic_case}, a quadratic Lyapunov function is constructed from the solution $P$ in \eqref{ww}. In the case of IQCs, such a quadratic function cannot serve as a Lyapunov function because it does not strictly decrease over all trajectories. Nevertheless, Theorem~\ref{thm:main} shows how $\rho$-hard IQCs can be used to certify a convergence rate and no Lyapunov function is explicitly constructed. We can explain this difference more explicitly. If $V(x)$ is a Lyapunov function, then by definition it satisfies the properties;
\begin{enumerate}[(i)]
\item $\lambda_1 \|x-x_\star\|^2 \le V(x) \le  \lambda_2 \|x-x_\star\|^2$ for all $x$ and $k$.
\item $V(x_{k+1}) \le \rho^2 V(x_k)$ for all system trajectories $\{x_k\}_{k\ge 0}$.
\end{enumerate}
Property (ii) implies that
\begin{equation}\label{bee}
V(x_k) \le \rho^{2k} V(x_0)
\end{equation}
which, combined with Property (i) implies that $\|x_k-x_\star\| \le \sqrt{\lambda_2/\lambda_1}\, \rho^k \|x_0-x_\star\|$. In Theorem~\ref{thm:main}, we use $V(x) = (x-x_\star)^\tp P (x-x_\star)$, which satisfies (i) but not (ii). So $V(x)$ is \emph{not} a Lyapunov function in the technical sense. Nevertheless, we prove directly that~\eqref{bee} holds, and so the desired result still holds. That is, $V(x)$ serves the same purpose as a Lyapunov function.

\subsection{IQCs for convex functions}\label{subsec:iqc_library}
We will derive three IQCs that are useful for describing gradients of strongly convex functions: the \emph{sector} (pointwise) IQC, the \emph{off-by-one} (hard) IQC, and \emph{weighted off-by-one} ($\rho$-hard) IQC.
%The sector IQC is also known as co-coercivity in optimization.
In general, gradients of strongly convex functions satisfy an infinite family of IQCs, originally characterized by Zames and Falb for the single-input-single-output case~\cite{ZamFal}. A generalization of the Zames-Falb IQCs to multidimensional functions is derived in~\cite{HeaWilZF}. Both the sector and off-by-one IQCs are special cases of Zames-Falb, while the weighted off-by-one IQC is a convex combination of the sector and off-by-one IQCs. While the Zames-Falb family is infinite, the three simple IQCs mentioned above are the only ones used in this paper. IQCs can be used to describe many other types of functions as well, and further examples are available in \cite{megrantzer}. We begin with some fundamental inequalities that describe strongly convex function.

\begin{prop}[basic properties]\label{prop:basic_properties}
Suppose $f\in\cvx$. Then the following properties hold for all $x,y\in\R^d$.
\begin{subequations}\label{cvx_properties}
\begin{gather}
f(y) \le f(x) + \grad f(x)^\tp (y - x) + \frac{L}{2} \|y-x\|^2 \label{eq:basicdef}\\
( \grad f(y)-\grad f(x) )^\tp (y - x) \ge \frac{1}{L}\, \| \grad f(y)-\grad f(x) \|^2 \label{eq:coercive}\\
f(y) \ge f(x) + \grad f(x)^\tp (y-x) + \frac{1}{2L}\|\grad f(y) - \grad f(x)\|^2 \label{eq:coercive2}\\
\bmat{y - x \\ \grad f(y) - \grad f(x)}^\tp \bmat{ -2mLI_d & (L+m)I_d \\ (L+m)I_d & -2I_d }
\bmat{y - x \\ \grad f(y) - \grad f(x)} \ge 0 \label{eq:sector}
\end{gather}
\end{subequations}
\end{prop}
\begin{proof}
Property~\eqref{eq:basicdef} follows from the definition of Lipschitz gradients. Properties~\eqref{eq:coercive} and~\eqref{eq:coercive2} are commonly known as \emph{co-coercivity}. To prove~\eqref{eq:sector}, define $g(x) \defeq f(x) - \tfrac{m}{2}\|x\|^2$ and note that $g \in \cvxx$. Applying~\eqref{eq:coercive} to $g$ and rearranging, we obtain
\[
(L+m)(\grad f(y)-\grad f(x))^\tp (y-x) \ge mL\|y-x\|^2 + \|\grad f(y)-\grad f(x)\|^2
\]
which is precisely~\eqref{eq:sector}. Detailed derivations of these properties can be found for example in~\cite{NesterovBook}.
\end{proof}

\begin{lem}[sector IQC]\label{lem:sector_iqc}
Suppose $f_k\in\cvx$ for each $k$, and $(y_\star,u_\star)$ is a common reference point for the gradients of $f_k$. In other words, $ u_\star = \grad f_k(y_\star)$ for all $k\ge 0$.
Let $\phi\defeq (\grad f_0,\grad f_1,\dots)$. If $u = \phi(y)$, then $\phi$ satisfies the \textbf{pointwise IQC} defined by
\begin{align*}
\Psi&= \bmat{LI_d & -I_d \\ -m I_d & I_d } &
&\text{and} &
M &= \bmat{ 0_d & I_d \\ I_d & 0_d }
\end{align*}
The corresponding quadratic inequality is that for all $y\in\ltwo^d$ and $k\ge 0$, we have
\begin{equation}\label{eq:sectiqc}
\bmat{y_k - y_\star \\ u_k-u_\star}^\tp \bmat{ -2mLI_d & (L+m)I_d \\ (L+m)I_d & -2I_d }
\bmat{y_k - y_\star \\ u_k-u_\star} \ge 0
\end{equation}
\end{lem}
\begin{proof}
Equation~\eqref{eq:sectiqc} follows immediately from~\eqref{eq:sector} by using $(f,x,y) \to (f_k,y_\star,y_k)$. It can be verified that
\[
\Psi^\tp M \Psi = \bmat{ -2mLI_d & (L+m)I_d \\ (L+m)I_d & -2I_d }
\qquad\text{and}\qquad
z_k - z_\star = \Psi \bmat{y_k-y_\star \\ u_k-u_\star}
\]
and therefore~\eqref{eq:sectiqc} is equivalent to $(z_k-z_\star)^\tp M (z_k-z_\star) \ge 0$ as required.
\end{proof}

\begin{rem}
In Lemma~\ref{lem:sector_iqc}, we use a slight abuse of notation in representing the map $\Psi: \ltwoe^d\times \ltwoe^d \to \ltwoe^m$. In writing $\Psi$ as a matrix in $\R^{2d \times 2d}$, we mean that $\Psi$ is a static map that operates pointwise on $(y,u)$. In other words,
\[
z_k = \Psi \bmat{y_k \\ u_k}
\qquad\text{for all $k$}.
\]
\end{rem}
\begin{lem}[off-by-one IQC]\label{lem:offbyone}
Suppose $f\in\cvx$ and $(y_\star,u_\star)$ is a reference for the gradient of $f$. In other words, $u_\star = \grad f(y_\star)$. Let $\phi \defeq (\grad f, \grad f,\dots)$. Then $\phi$ satisfies the \textbf{hard IQC} defined by
\begin{align*}
\Psi &= \left[\begin{array}{c|cc}
0_d & -L I_d & I_d \\ \hlinet
I_d & LI_d & -I_d \\
0_d & -m I_d & I_d
\end{array}\right]
&&\text{and}
&
M &= \bmat{ 0_d & I_d \\ I_d & 0_d }
\end{align*}
The corresponding quadratic inequality is that for all $y\in\ltwo^d$ and $k\ge 0$, we have
\begin{equation}\label{eq:popov1}
(\tilde u_0-m\tilde y_0)^\tp (L\tilde y_0-\tilde u_0)
+ \sum_{t=1}^k
(\tilde u_t-m\tilde y_t)^\tp \bigl( L(\tilde y_t-\tilde y_{t-1}) - (\tilde u_t-\tilde u_{t-1}) \bigr) \ge 0
\end{equation}
where we have defined $\tilde y_k \defeq y_k-y_\star$ and $\tilde u_k \defeq u_k-u_\star$.
\end{lem}
\begin{proof}
Define the function
\[
g(x) \defeq f(x)-f(y_\star) -  \frac{m}{2}\|x-y_\star\|^2
\] 
It is straightforward to check that $g \in \cvxx$, and $g(x) \ge g(y_\star) =0$ for all $x\in\R^d$. Applying~\eqref{eq:coercive2} using $(f,x,y)\to (g,y_\star,y_k)$, we observe that
\begin{equation}\label{turtle}
q_k \defeq (L-m)g(y_k) - \frac12 \|\grad g(y_k) \|^2 \ge 0
\qquad\text{for all $k\ge 0$}
\end{equation}
Moreover, $\grad g(y_k) = \grad f(y_k) -m(y_k-y_\star) = \tilde u_k-m \tilde y_k$. Therefore, we may manipulate the first term in~\eqref{eq:popov1} to eliminate $\tilde u_0$ and obtain
\begin{align}\label{eq:q0}
(\tilde u_0-m\tilde y_0)^\tp (L\tilde y_0-\tilde u_0)
&=\grad g(y_0)^\tp((L-m) \tilde y_0 - \grad g(y_0)) \notag \\
&= (L-m) \grad g(y_0)^\tp \tilde y_0 - \|\grad g(y_0)\|^2 \notag\\
&\ge (L-m) g(y_0) - \tfrac{1}{2}\|\grad g(y_0) \|^2 \notag\\
&= q_0
\end{align}
where the inequality follows from applying~\eqref{eq:coercive2} using $(f,x,y)\to (g,y_0,y_\star)$.
Similarly, the $t^\text{th}$ term in the sum in~\eqref{eq:popov1} can be bounded by eliminating $\tilde u_t$ and $\tilde u_{t-1}$.
\begin{align}\label{eq:qt}
&\hspace{-2cm}(\tilde u_t -m\tilde y_t)^\tp ( L(\tilde y_t-\tilde y_{t-1}) - (\tilde u_t-\tilde u_{t-1}) ) \notag\\
&= (L-m)\grad g(y_t)^\tp (\tilde y_t-\tilde y_{t-1}) - \grad g(y_t)^\tp(\grad g(y_t)-\grad g(y_{t-1}) ) \notag\\
&\ge (L-m)(g(y_t)-g(y_{t-1})) - \tfrac{1}{2}\|\grad g(y_t)\|^2 +  \tfrac{1}{2}\|\grad g(y_{t-1})\|^2 \notag \\
&= q_t-q_{t-1}
\end{align}
where the inequality follows this time from applying~\eqref{eq:coercive2} using $(f,x,y)\to (g,y_t,y_{t-1})$.
Substituting~\eqref{eq:q0} and~\eqref{eq:qt} into the left-hand side of~\eqref{eq:popov1}, the sum telescopes and we obtain the lower bound $q_k$, which is nonnegative from~\eqref{turtle}.

To verify the IQC factorization, note that the state equations for $\Psi$  given in the statement of Lemma~\ref{lem:offbyone} are
\[
\left\{
\begin{aligned}
\zeta_0 &= \zeta_\star \\
\zeta_{k+1} &= -Ly_k +  u_k \\
z_k &= \bmat{ \zeta_k + L y_k -  u_k \\ -m  y_k +  u_k}
\end{aligned}\right\}
\implies
\left\{
\begin{aligned}
z_0 &= \bmat{\zeta_\star + L y_0 -  u_0 \\ -m y_0 + u_0} \\
z_k &= \bmat{L( y_k- y_{k-1}) -( u_k -  u_{k-1}) \\ -m y_k +  u_k },\quad k\ge 1
\end{aligned}\right\}
\]
Moreover, the solution to the fixed-point equations~\eqref{psi0} are
\begin{align*}
\zeta_\star &= -Ly_\star + u_\star
&&\text{and}&
z_\star &= \bmat{ 0 \\ -my_\star + u_\star}
\end{align*}
Therefore, we conclude that
\[
z_0-z_\star = \bmat{ L\tilde y_0 - \tilde u_0 \\ -m\tilde y_0 + \tilde u_0 }
\quad\text{and}\quad
z_k-z_\star = \bmat{L( \tilde y_k- \tilde y_{k-1}) -( \tilde u_k -  \tilde u_{k-1}) \\ -m \tilde y_k +  \tilde u_k },\, k\ge 1
\]
and it follows that $\sum_{t=0}^{k}(z_t-z_\star)^\tp M (z_t-z_\star) \ge 0$ is equivalent to~\eqref{eq:popov1}, as required.
\end{proof}

Note that the sector IQC~\eqref{eq:sectiqc} is a special case of the off-by-one IQC when $k=0$. The off-by-one IQC is itself a special case of the Zames-Falb IQC, which we now describe.

\begin{lem}[Zames-Falb IQC]\label{lem:zames-falb}
Suppose $f\in\cvx$ has the optimal point $u_\star = \grad f(y_\star) = 0$. Let $\phi \defeq (\grad f, \grad f,\dots)$ and let $h_1,h_2,\dots$ be any sequence of real numbers that satisfies
\begin{enumerate}[(i)]
\item $\{h_\tau\}_{\tau\ge 1}$ is finitely nonzero, and $h_s$ is the last nonzero component.
\item $0 \le h_\tau \le 1$ for all $\tau \ge 1$.
\item $\sum_{\tau=1}^\infty h_\tau \le 1$.
\end{enumerate}
Then $\phi$ satisfies the \textbf{hard IQC} defined by
\begin{align*}
\Psi &= \left[\begin{array}{cccc|cc}
0_d & 0_d & \dots & 0_d   &   -LI_d & I_d \\
I_d & 0_d & \dots & 0_d   &   0_d & 0_d  \\
\vdots & \ddots & \ddots & \vdots   &    \vdots & \vdots \\
0_d & \dots & I_d & 0_d   &   0_d & 0_d \\ \hlinet
h_1I_d & h_2I_d & \dots & h_sI_d &   LI_d & -I_d \\
0_d & 0_d & \dots & 0_d  &  -mI_d & I_d
\end{array}\right]
&&\text{and}&
M &= \bmat{ 0_d & I_d \\ I_d & 0_d }
\end{align*}
The corresponding quadratic inequality is that for all $y\in\ltwo^d$ and $k\ge 0$, we have
\begin{equation}\label{eq:zames-falb1}
 \sum_{t=0}^k
(\tilde u_t - m\tilde y_t)^\tp \left( L \bbbl(\tilde y_t-\sum_{\tau=1}^t h_\tau \tilde y_{t-\tau}\bbbr) - \bbbl(\tilde u_t- \sum_{\tau=1}^t h_\tau \tilde u_{t-\tau}\bbbr) \right) \ge 0 
\end{equation}
where we have defined $\tilde y_k \defeq y_k-y_\star$ and $\tilde u_k \defeq u_k-u_\star$.
\end{lem}
\begin{proof} We will construct a proof for a general sequence $h_1,h_2,\dots$ by first considering a specific set of sequences. Fix some $j\ge 1$ and consider the case where
\[
h_\tau = \begin{cases} 1 & \tau = j \\ 0 & \tau \ne j \end{cases}
\]
For $t< j$, the terms in the sum~\eqref{eq:zames-falb1} have the form
\[
(\tilde u_t-m\tilde y_t)^\tp ( L \tilde y_t - \tilde u_t)
\]
which are bounded below by $q_t \ge 0$, as proven in Lemma~\ref{lem:offbyone}, \eqref{turtle}--\eqref{eq:q0}. For $t\geq j$, the terms in the sum~\eqref{eq:zames-falb1} have the form
\begin{align*}
(\tilde u_t - m\tilde y_t)^\tp ( L (\tilde y_t- \tilde y_{t-j}) - (\tilde u_t-  \tilde u_{t-j}) )
\end{align*}
which are bounded below by $q_t - q_{t-j}$, as proven in~\eqref{eq:qt}. Summing up~\eqref{eq:zames-falb1} for all $t$ yields a telescoping sum, thereby proving that~\eqref{eq:zames-falb1} holds. This can be thought of an ``off-by-$j$'' IQC. Indeed, when $j=1$, we recover the off-by-one IQC of Lemma~\ref{lem:offbyone}.

Now note that if we take a convex combination of the inequalities~\eqref{eq:zames-falb1} corresponding to each off-by-$j$ IQC and let the associated coefficient be $h_j$, we have proven~\eqref{eq:zames-falb1} for the case of a general sequence $h_1,h_2,\dots$.
\end{proof}

Though we will not make use of the more general Zames-Falb family of inequalities, we include them as they are interesting in their own right and may find applications in future work.
%In particular, the Zames-Falb family completely describes strongly convex functions in the following sense. If $\phi$ satisfies every Zames-Falb IQC for a particular $m$ and $L$, then $\phi = (\grad f,\grad f,\dots)$ for some $f \in \cvx$~\cite{Carrasco}.
%
We conclude this section with a $\rho$-hard version of the off-by-one IQC. This final IQC will be critical for deriving convergence rates.

\begin{lem}[weighted off-by-one IQC]\label{lem:combo}
Suppose $f\in\cvx$ and $(y_\star,u_\star)$ is a reference for the gradient of $f$. In other words, $u_\star = \grad f(y_\star)$. Let $\phi \defeq (\grad f, \grad f,\dots)$. Then for any $(\bar \rho,\rho)$ satisfying $0 \le \bar\rho \le \rho \le 1$, $\phi$ satisfies the \textbf{$\rho$-hard IQC} defined by
\begin{align*}
\Psi &= \left[\begin{array}{c|cc}
0_d & -L I_d & I_d \\ \hlinet
\bar\rho^2 I_d & LI_d & -I_d \\
0_d & -m I_d & I_d
\end{array}\right]
&&\text{and}
&
M &= \bmat{ 0_d & I_d \\ I_d & 0_d }
\end{align*}
The corresponding quadratic inequality is that for all $y\in\ltwo^d$ and $k\ge 0$, we have
\begin{equation}\label{eq:popov2}
(\tilde u_0 - m \tilde y_0)^\tp (L\tilde y_0-\tilde u_0)
+ \sum_{t=1}^k
\rho^{-2t}
(\tilde u_t-m\tilde y_t)^\tp \bigl( L( \tilde y_t-\bar\rho^2 \tilde y_{t-1}) - ( \tilde u_t - \bar\rho^2 \tilde u_{t-1}) \bigr) \ge 0
\end{equation}
where we have defined $\tilde y_k \defeq y_k-y_\star$ and $\tilde u_k \defeq u_k-u_\star$.
\end{lem}
\begin{proof}
Note that the weighted off-by-one IQC is a Zames-Falb IQC with $h = (\bar\rho^2,0,\dots)$. Thus the hardness and the factorization $(\Psi,M)$ follows from Lemma~\ref{lem:zames-falb}. In order to prove $\rho$-hardness~\eqref{eq:popov2}, a bit more work is required. First, observe (see remarks on pointwise and hard IQCs after Theorem~\ref{thm:main}) that it suffices to show $\bar\rho$-hardness, and this will imply $\rho$-hardness. The $t^\text{th}$ term in the sum in~\eqref{eq:popov2} can be bounded as follows. First, define the general terms in the sector (Lemma~\ref{lem:sector_iqc}) and off-by-one (Lemma~\ref{lem:offbyone}) inequalities:
\begin{align*}
s_t &\defeq (\tilde u_t-m\tilde y_t)^\tp ( L\tilde y_t - \tilde u_t)\\
p_t &\defeq (\tilde u_t-m\tilde y_t)^\tp \bigl( L( \tilde y_t-\tilde y_{t-1}) - ( \tilde u_t - \tilde u_{t-1}) \bigr)
\end{align*}
Algebraic manipulations reveal that the general term in the sum~\eqref{eq:popov2} satisfies
\begin{align*}
(\tilde u_t-m\tilde y_t)^\tp \bigl( L( \tilde y_t-\bar\rho^2 \tilde y_{t-1}) - ( \tilde u_t - \bar\rho^2 \tilde u_{t-1}) \bigr)
&= (1-\bar\rho^2)s_t + \bar\rho^2 p_t \\
&\ge (1-\bar\rho^2)q_t + \bar\rho^2 (q_t - q_{t-1}) \\
&= q_t - \bar\rho^2 q_{t-1}
\end{align*}
where the inequalities follow from~\eqref{eq:q0} and~\eqref{eq:qt}. Substituting the general term back into~\eqref{eq:popov2} with $\rho = \bar\rho$, the $\bar\rho^{-2t}$ coefficient causes the sum to telescope and we are left with $\bar\rho^{-2k}q_k$, which is nonnegative from~\eqref{turtle}. This completes the proof.
\end{proof}

\begin{rem}\label{rem:wobo}
In implementing the weighted off-by-one IQC, one can simply set $\bar \rho = \rho$. However, a less conservative approach is to keep $\bar \rho$ as an additional degree of freedom. In Theorem~\ref{thm:main}, the IQC constraint is included in~\eqref{eq:expLMI} in the final term and is multiplied by the constant $\lambda \ge 0$. When using the weighted off-by-one IQC, this amounts to:
\begin{align*}
\lambda \left( (1-\bar \rho^2) s_t + \bar \rho^2 p_t \right)
\qquad \text{with the constraints: } 0 \le \bar \rho \le \rho
\text{ and }\lambda \ge 0
\end{align*}
By defining $\lambda_1 = \lambda(1-\bar\rho^2)$ and $\lambda_2 = \lambda\bar\rho^2$, an equivalent expression is
\[
\lambda_1 s_t + \lambda_2 p_t
\qquad\text{with the constraints: } \lambda_1,\lambda_2 \ge 0
\text{ and } \lambda_2 \le \rho^2(\lambda_1+\lambda_2)
\]
\end{rem}

%%%%%%%%%%%%%%%%%%%%%%%%%%%%%%%%%%%%%%%%%%%%%%%%%%%%%%%%%%%%%%%%%%%%%%%%%%%%%%%

\subsection{Historical context of IQCs and Lyapunov theory}\label{sec:history_lyap}
Constructing Lyapunov functions has a long history in control and dynamical systems, and the central focus of this paper is borrowing tools from this literature to see how we can generalize our analysis from quadratic functions to more general, nonlinear convex functions.

One of the most fundamental problems in control theory is certifying the stability of nonlinear systems.
In interconnected systems such as electric circuits or chemical plants, individual components are typically modeled using differential (or difference) equations.  Interconnected systems often contain nonlinearities or components that are otherwise difficult to model. The earliest results on such systems date back to the work of Lur'e and Postnikov~\cite{lure_postnikov}. The goal was to prove stability under a wide range of admissible uncertainties. This notion of robust stability was called \emph{absolute stability}.  Indeed, Lur'e studied precisely the model we are concerned with: a known linear system interconnected in feedback to an uncertain nonlinear system.

In the 1960's and 70's, several sufficient conditions for absolute stability were expressed as frequency-domain conditions. In other words, the main objects of interest are ratios of the Laplace transforms of the outputs to the inputs, also known as \emph{transfer functions}. Examples include the Popov criterion~\cite{popov}, the small-gain theorem, the circle criterion, and passivity theory~\cite{zames}. Frequency-domain conditions were popular at the time because they could be verified graphically. The work of Willems~\cite{willems} unified many of the existing results by casting them in the time domain in a framework called \emph{dissipativity theory}. This notion is on one hand a generalization of Lyapunov functions to include systems with exogenous inputs, and on the other hand a generalization of passivity theory and the small-gain theorem. These ideas form the core of modern nonlinear control theory, and are covered in many textbooks such as Khalil~\cite{khalil}.

With the advent of computers, graphical methods were no longer required. The connection between frequency-domain conditions and Linear Matrix Inequalities (LMIs) was made by Kalman~\cite{kalman} and Yakubovich~\cite{yakubovich} and culminated in the Kalman-Yakubovich-Popov (KYP) lemma, also known as the Positive-Real lemma. This paved the way for the use of modern computational tools such as semidefinite programming. Another important development is the concept of the \emph{structured singular value}~\cite{doyle_ssv}, also known as $\mu$-analysis. While previous theory had been used to describe \emph{static} nonlinearities or uncertainties, $\mu$-analysis is a computationally tractable framework for describing a system containing multiple \emph{dynamic} uncertainties. A survey of $\mu$-related techniques and results is given in~\cite{packard_doyle_mu}. For a comprehensive overview of the history and development of LMIs in control theory, we refer the reader to~\cite{boydLMI}.

Integral Quadratic Constraints (IQCs) were first introduced by Yakubovich, who considered the notion of imposing quadratic constraints on an infinite-horizon control problem~\cite{yakubovich_iqc}, and combining multiple constraints via the S-procedure~\cite{yakubovich_sprocedure}. The definitive work on IQCs is Megretski and Rantzer~\cite{megrantzer}. In this seminal paper, the authors showed that dissipativity theory, as well as all the frequency-domain conditions, could be formulated as IQCs. Furthermore, the KYP lemma in conjunction with the S-procedure allows stability to be verified by solving an LMI.

The seminal paper on IQCs~\cite{megrantzer} develops the theory primarily in the frequency domain, but also alludes to time-domain versions of the results by introducing \emph{hard} IQCs. This notion of hard IQCs is pursued in~\cite{seiler13TAC}, where the main IQC stability theorem is rederived entirely in the time domain. In the time domain, these constraints parallel the development of Nesterov, where we are able to construct inequalities linking multiple inputs and outputs of uncertain functions.  This allows us to provide a wholly self-contained development of the theory.  Moreover, we are able to enhance the techniques of~\cite{seiler13TAC}, providing new IQCs and considerably sharper rates of convergence than those discussed in the earlier work. In this sense, our work provides useful methods for control theorists interested in estimating rates of stabilization of their control systems.

%%%%%%%%%%%%%%%%%%%%%%%%%%%%%%%%%%%%%%%%%%%%%%%%%%%%%%%%%%%%%%%%%%%%%%%%%%%%%%%
\section{Case studies}\label{sec:computation}

We now use the results of Section~\ref{sec:iqc} to rederive some existing results from the literature on iterative large-scale algorithms. The IQC approach gives a unified method to analyze many different algorithms. In addition to verifying existing results, we also present a negative result that was not previously known.

\subsection{Computational approach}\label{subsec:comp_approach}
Given an iterative algorithm, our first step is to express it as a feedback interconnection of a discrete linear time-invariant dynamical system with a nonlinearity representing $\grad f$. This procedure is explained in Section~\ref{sec:opt-dyn} and yields matrices $(A,B,C)$.

The next step is to decide which IQCs will be used to characterize the nonlinearity. A simple but conservative choice is the sector IQC defined in Lemma~\ref{lem:sector_iqc}. A less conservative choice is the weighted off-by-one IQC of Lemma~\ref{lem:combo}. For the chosen $(\Psi,M)$, we find the smallest $\rho$ such that the semidefinite program (SDP)~\eqref{eq:expLMI} of Theorem~\ref{thm:main} is feasible. In the case of the sector IQC, the SDP has variables $(P,\lambda,\rho)$. For the weighted off-by-one IQC, the SDP has variables $(P,\lambda_1,\lambda_2,\rho)$ as explained in Remark~\ref{rem:wobo}.
The resulting $\rho$ is an upper bound for the worst-case convergence rate of the algorithm. Specifically,
\[
\|\xi_k-\xi_\star\| \le \sqrt{\cond(P)}\, \rho^k \|\xi_0-\xi_\star\|.
\]

To solve the SDP~\eqref{eq:expLMI} numerically, observe that it is a quasiconvex program. In particular, for every fixed $\rho$, \eqref{eq:expLMI} is an LMI. The simplest way to solve~\eqref{eq:expLMI} is to use a bisection search on $\rho$. For a fixed $\rho$, the SDP~\eqref{eq:expLMI} or \eqref{eq:mainLMI2} become an LMI and can be efficiently solved using interior-point methods. Popular implementations include SDPT3, SeDuMi, and Mosek. This approach was used for all the simulations presented herein.

More sophisticated methods exist to solve~\eqref{eq:expLMI} as well. A quasiconvex program of the type~\eqref{eq:expLMI} is known as a  \emph{generalized eigenvalue optimization problem} (GEVP) \cite{boydLMI}. The GEVP is well-studied and modified interior-point methods such as the \emph{method of centers}~\cite{boyd_centers} and the \emph{long-step method of analytic centers}~\cite{nemirovski1997long} can be used to solve it.

\subsection{Lossless dimensionality reduction}\label{subsec:dim_reduction}

The size of the SDP in~\eqref{eq:expLMI} is proportional to $d$, the size of the state $\xi_k$ in the optimization algorithm. This can be problematic in cases where $d$ is large because it can be computationally costly to solve large SDPs. In many cases of interest, however, the algorithms we wish to analyze have a block-diagonal structure. For example, Nesterov's accelerated method has the form~\eqref{bb}, which is
\begin{equation}\label{bb2}
\left[\begin{array}{c|c}
    A & B \\ \hlinet
    C & D
\end{array}\right] =\left[\begin{array}{cc|c}
    (1+\beta)I_d & -\beta I_d  & -\alpha I_d\\
    I_d & 0_d & 0_d\\\hlinet
    (1+\beta)I_d & -\beta I_d & 0_d
\end{array}\right]
\end{equation}
Each of the matrices $(A,B,C)$ is a block matrix with repeated diagonal blocks. Using Kronecker product notation (see Section~\ref{sec:notation}, this means for example that
\[
A = \bmat{ 1+\beta & -\beta \\ 1 & 0 } \otimes I_d
\]
and similarly for $B$ and $C$. Moreover, the IQCs we use to describe $\grad f$ have the same sort of structure. That is, $(A_\Psi,B_\Psi^y,B_\Psi^u,C_\Psi,D_\Psi^y,D_\Psi^u)$ are block matrices with repeated diagonal blocks. Now consider the SDP~\eqref{eq:expLMI} from Theorem~\ref{thm:main}.
\begin{equation}\label{eq:expLMI2}
\bmat{ \hat A^\tp P \hat A - \rho^2 P & \hat A^\tp P \hat B \\
\hat B^\tp P \hat A & \hat B^\tp P \hat B} + \lambda
 \bmat{\hat C & \hat D}^\tp M \bmat{\hat C & \hat D}
\preceq 0
\end{equation}
Based on the discussion above, each of the matrices $(\hat A,\hat B,\hat C,\hat D,M)$ have the form e.g. $A_0 \otimes I_d$. Rather than looking for a general $P \in \R^{nd\times nd}$ with $P \succ 0_{nd}$, if we restrict our search to $P = P_0 \otimes I_d$ with $P_0\in\R^{n\times n}$ and $P_0\succ 0_n$, then the SDP reduces to
\begin{equation}\label{eq:expLMI3}
\bmat{ \hat A_0^\tp P_0 \hat A_0 - \rho^2 P_0 & \hat A_0^\tp P_0 \hat B_0 \\
\hat B_0^\tp P_0 \hat A_0 & \hat B_0^\tp P_0 \hat B_0} + \lambda
 \bmat{\hat C_0 & \hat D_0}^\tp M_0 \bmat{\hat C_0 & \hat D_0}
\preceq 0
\end{equation}
The resulting SDP no longer depends on $d$ and is effectively the same as if we had solved the original problem with $d=1$. As it turns out, there is no loss of generality in assuming a $P$ of this form. To see why this is so, first suppose $P_0\succ 0$ satisfies~\eqref{eq:expLMI3}. Then clearly $P=P_0\otimes I_d$ satisfies~\eqref{eq:expLMI2}. Conversely, suppose $P\succ 0$ satisfies~\eqref{eq:expLMI2}. Then define the matrix
$
P_0 \defeq (I_n \otimes e_1)^\tp P(I_n \otimes e_1)
$
where $e_1 = \bmat{1 & 0 & \dots & 0}^\tp \in \R^{d\times 1}$. Note that $P_0$ is an $n\times n$ principal submatrix of $P$, and therefore $P_0\succ 0$ because $P \succ 0$. Multiplying the left-hand side of~\eqref{eq:expLMI2} by $(I_n \otimes e_1)^\tp$ on the left and $(I_n \otimes e_1)$ on the right, we conclude that $P_0$ satisfies~\eqref{eq:expLMI3}. Thus, $\hat P = P_0 \otimes I_d$ is also a solution to~\eqref{eq:expLMI2}. In other words, \eqref{eq:expLMI2} is feasible if and only if~\eqref{eq:expLMI3} is feasible.

\subsection{Known bounds for first-order optimization algorithms}

The following proposition summarizes some of the known bounds for optimizing strongly convex functions.

\begin{prop} \label{prop:strongly-convex-rates} The following table gives worst-case rate bounds for different algorithms and parameter choices when applied to a class of \textbf{strongly convex functions}. We assume here that $f:\R^d \to \R$ where $f\in\cvx$. Again, we define $\kappa\defeq L/m$.

\begin{center}
	\smallskip
	\begin{tabular}{|l|l|l|l|}
		\hlinet
		Method & Parameter choice & Rate bound & Comment \\ \hlinet\rule{0pt}{3.0ex}%
		Gradient & $\alpha = \frac{1}{L}$ & $\rho \le \sqrt{\frac{\kappa-1}{\kappa+1}}$ & popular choice \\
		Nesterov & $\alpha = \frac{1}{L},\, \beta = \frac{\sqrt{\kappa}-1}{\sqrt{\kappa}+1}$ & $\rho \le \sqrt{1-\frac{1}{\sqrt{\kappa}}}$ & standard choice \\ \hlinet
		Gradient & $\alpha = \frac{2}{L+m}$ & $\rho = \frac{\kappa-1}{\kappa+1}$ & optimal tuning \\[1mm]\hline
	\end{tabular}
	\smallskip
\end{center}
\end{prop}
The Gradient bounds in the table above follow from the bound $\rho \le \sqrt{1-\frac{2\alpha m L}{L+m}}$, which is proven in~\cite{NesterovBook}. A tighter Gradient bound $\rho \le \max \bl\{ |1-\alpha m|, |1-\alpha L| \br\}$ is proven in~\cite{polyak1987introduction} but makes the additional assumption that $f$ is twice differentiable. The Nesterov bound in Proposition~\ref{prop:strongly-convex-rates} is proven in~\cite{NesterovBook} using the technique of estimate sequences. There are no known global convergence guarantees for the Heavy-ball method in the case of strongly convex functions, but it is proven in~\cite{polyak1987introduction} that the Heavy-ball method converges~\emph{locally} with the same rate as in Proposition~\ref{prop:quadratic-rates}.

In the following sections, we will use IQC machinery to demonstrate that the first two bounds in Proposition~\ref{prop:strongly-convex-rates} are loose.  We will construct tighter bounds for the strongly convex case without requiring additional assumptions about locality or twice-differentiability.   We will then use our framework to help guide a refutation of the convergence of the Heavy-ball method.

\subsection{The Gradient method}\label{grad}
The Gradient method with constant stepsize is among the simplest optimization schemes. The recursion is given by
\begin{equation}\label{eq:grad_req}
\xi_{k+1} = \xi_k - \alpha\grad f(\xi_k)
\end{equation}
We will analyze this algorithm by applying Theorem~\ref{thm:main}. Since $f\in\cvx$, we may use the sector IQC of Lemma~\ref{lem:sector_iqc} and \eqref{eq:expLMI} together with the dimensionality reduction of Section~\ref{subsec:dim_reduction} yields the following SDP. 
\begin{equation}\label{eq:ex_lmix}
\bmat{(1-\rho^2)P & -\alpha P \\ -\alpha P & \alpha ^2P} + \lambda \bmat{ -2mL & (L+m) \\ (L+m) & -2} \preceq 0,
\qquad P \succ 0,\qquad
\lambda \ge 0
\end{equation}
Note that $P$ is $1\times 1$, so we may set $P=1$ without loss of generality and we obtain the following LMI in $(\rho^2,\lambda)$.
\begin{equation}\label{eq:ex_lmixx}
\bmat{1-\rho^2 & -\alpha  \\ -\alpha  & \alpha ^2} + \lambda \bmat{ -2mL & L+m \\ L+m & -2} \preceq 0
\qquad\text{and}\qquad
\lambda \ge 0
\end{equation}
Using Schur complements,~\eqref{eq:ex_lmixx} is equivalent to
\begin{equation}\label{eq:ex_almosttherex}
\lambda \ge \frac{\alpha^2}{2}
\qquad\text{and}\qquad
\rho^2 \ge 1-2mL\lambda - \frac{ (\alpha - (L+m)\lambda)^2}{ 2\lambda - \alpha^2}
\end{equation}
By analyzing the lower bound on $\rho$ in~\eqref{eq:ex_almosttherex}, we can find the optimal choice of $\lambda$ as a function of the stepsize $\alpha$. Omitting the details, we eventually obtain the simple expression $\rho = \max \bl\{ |1-\alpha m|, |1-\alpha L| \br\}$. This is precisely the bound found for the quadratic case, as derived in Appendix~\ref{A:prop1proof}. However, we have shown something much stronger here, since the only assumption we made about $f$ is that $\grad f$ satisfies the sector IQC of Lemma~\ref{lem:sector_iqc}. In particular, the Gradient method rates in Proposition~\ref{prop:quadratic-rates} hold not only for quadratics, but also for strongly convex functions, and even for functions that change or switch over time (either stochastically, adversarially, or otherwise), so long as each function satisfies the pointwise sector constraint.
Note that~\eqref{eq:ex_lmixx} can be transformed using Schur complements:
\begin{equation}\label{eq:lmifull}
\bmat{ -2mL\lambda-\rho^2 & (L+m)\lambda & 1 \\ (L+m)\lambda & -2\lambda & -\alpha  \\ 1 & -\alpha  & -1 } \preceq 0
\end{equation}
And now~\eqref{eq:lmifull} is linear in $(\rho^2,\lambda,\alpha )$. This formulation allows one to directly answer questions such as ``what range of stepsizes can yield a given rate?''.

\subsection{Nesterov's accelerated method}\label{nest}

Nesterov's accelerated method with constant stepsize converges at a linear rate. There exists some $c>0$ such that for any initial condition $\xi_0$,
\[
\|\xi_{k}-\xi_\star\| \leq c \rho^{k} \|\xi_0-\xi_\star\|
\qquad\text{with}\qquad
\rho = \sqrt{1 - \sqrt{\tfrac{m}{L}}}
\]
when applied to functions $f\in\cvx$. In this case, the parameters are the standard parameters from Proposition~\ref{prop:strongly-convex-rates}, which are $\alpha \defeq 1/L$ and $\beta \defeq (\sqrt{L}-\sqrt{m})/(\sqrt{L}+\sqrt{m})$ \cite{NesterovBook}. Nesterov also showed that a \emph{lower bound} on convergence rate for \emph{any} algorithm of the form~\eqref{eq:lure} and for any $f \in \cvx$ is given by
\begin{equation}\label{rho_opt}
\|\xi_{k}-\xi_\star\| \geq \rho_\textup{opt}^{k} \|\xi_0-\xi_\star\|
\qquad\text{with}\qquad
\rho_\textup{opt} = \frac{\sqrt{L}-\sqrt{m}}{\sqrt{L}+\sqrt{m}}\,.
\end{equation}
Since $\rho$ and $\rho_\textup{opt}$ behave similarly as $L/m \to \infty$, Nesterov's accelerated method is sometimes called ``optimal'' or ``nearly optimal''.
%
%in the sense that there is no algorithm of the form~\eqref{eq:lure} that converges with a rate faster than
%\begin{equation}\label{rho_opt}
%\rho_\textup{opt} = \frac{\sqrt{L}-\sqrt{m}}{\sqrt{L}+\sqrt{m}}\,.
%\end{equation}

We computed the rate bounds using Theorem~\ref{thm:main} using either the sector IQC of Lemma~\ref{lem:sector_iqc}, or a combination of the sector IQC and the weighted off-by-one IQC of Lemma~\ref{lem:combo}. It is important to note that unlike the Gradient method case, the LMI~\eqref{eq:expLMI} is no longer linear in $\rho^2$. Therefore, we found the minimal $\rho$ by performing a bisection search on $\rho$, see the first plot in Figure~\ref{fig:Nesterov_rate}.
\begin{figure}[ht]
\centering
\includegraphics[width=.8\linewidth]{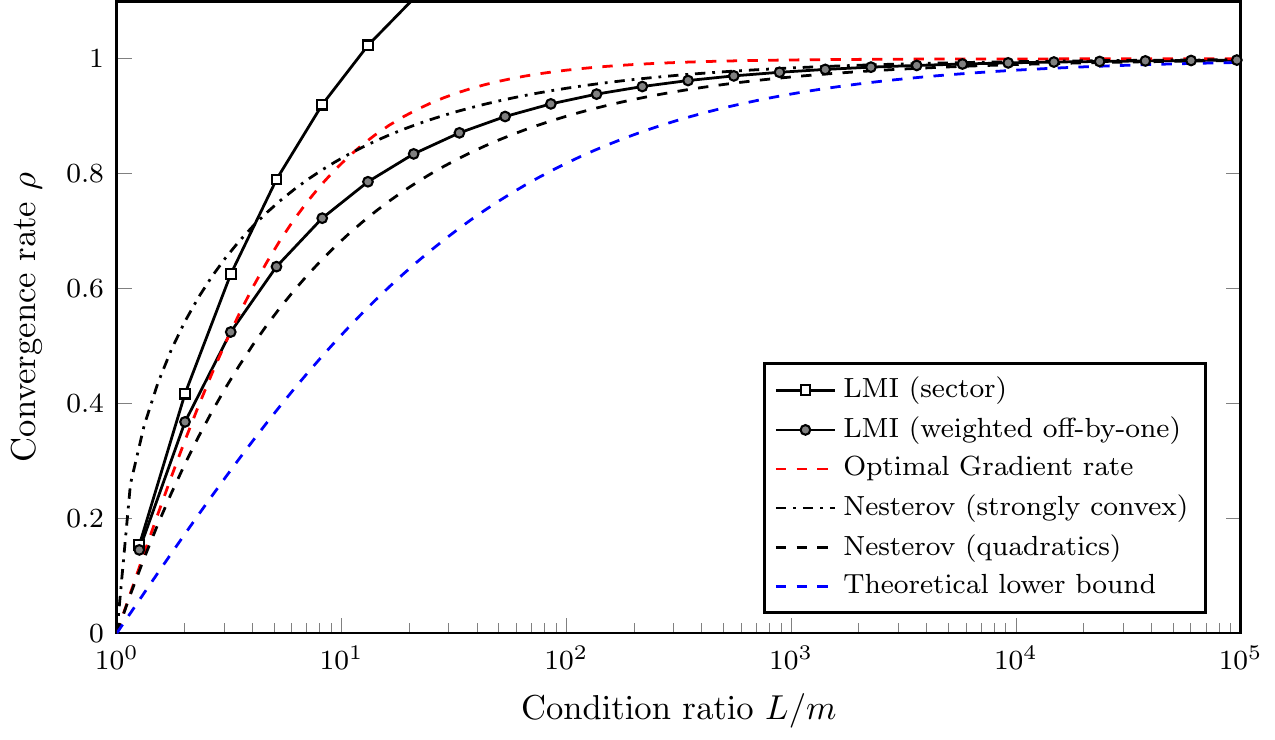}
\includegraphics[width=.8\linewidth]{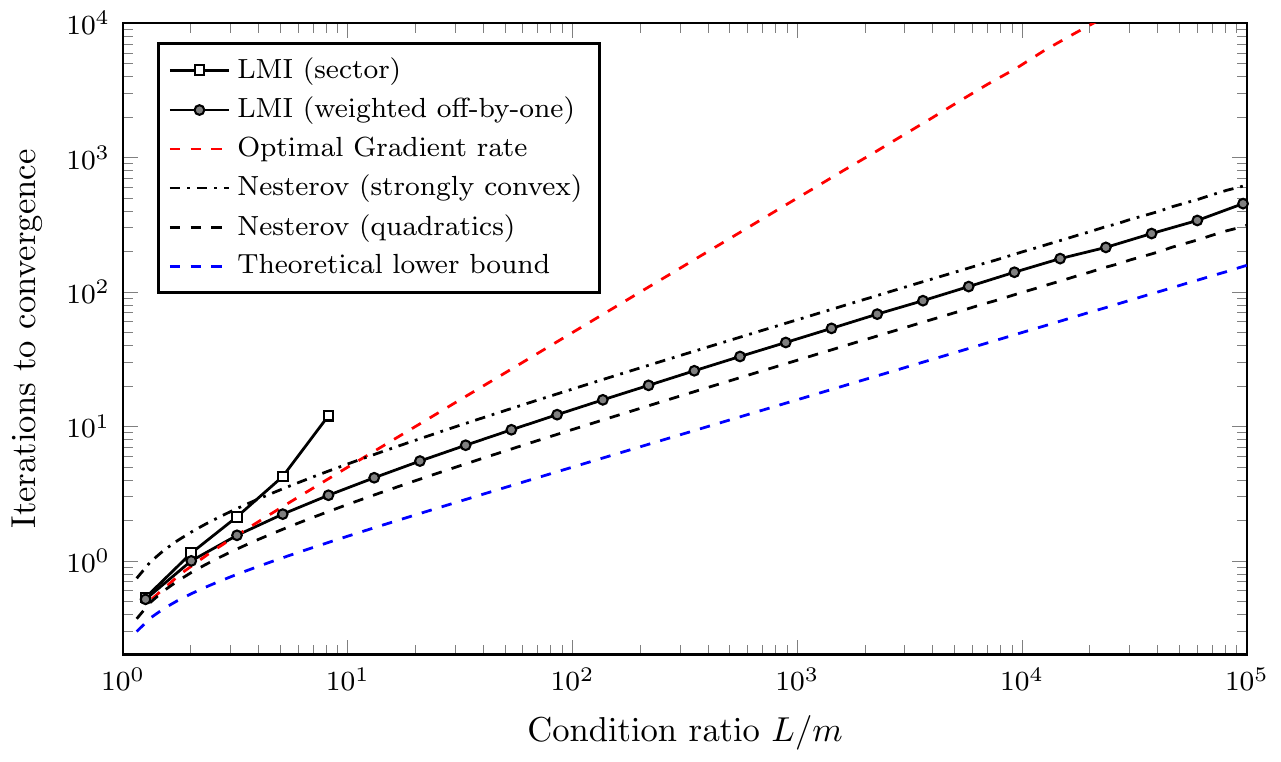}
\caption{Upper bounds for Nesterov's accelerated method applied to $f\in\cvx$ using the standard tuning in Proposition~\ref{prop:strongly-convex-rates}. We tested both the sector IQC and the weighted off-by-one IQC. The first plot shows convergence rate and the second plot shows number of iterations required to achieve convergence to a specified tolerance. The \emph{theoretical lower bound} $\rho_\textup{opt}$ is given in~\eqref{rho_opt}. The rate that can be certified using the LMI approach is strictly better than the rate proved in~\cite{NesterovBook} using estimate sequences.\label{fig:Nesterov_rate}}
\end{figure}

The rate obtained using the sector IQC alone is very poor. To understand why, recall from Lemma~\ref{lem:sector_iqc} that the sector IQC allows for $f_k$ to be different at each iteration. Unlike the Gradient method, Nesterov's accelerated method is not robust to having a changing $f_k$. However, convergence can nevertheless be guaranteed as long as $\rho < 1$, which corresponds approximately to $L/m < 11.7$.

The rate obtained using the weighted off-by-one IQC improves upon the rate proven in~\cite{NesterovBook} using the estimate sequence approach (see Proposition~\ref{prop:strongly-convex-rates}). Note that we do not have an analytical expression for the improved bound; it was found numerically by solving the LMI of Theorem~\ref{thm:main}.

Given that $\|x_k\| \le \sqrt{\cond(P)} \rho^{k} \|x_0\|$, if we seek the smallest $k$ such that $\|x_k\| \le \epsilon$, then it suffices that $\sqrt{\cond(P)} \rho^{k} \|x_0\| \le \epsilon$. This implies that
\begin{equation}\label{dolphin}
k \ge  \left( -\frac{1}{2\log\rho}\right) \log\left(\frac{\cond(P)\|x_0\|^2}{\epsilon^2}\right)
\end{equation}
For the second plot in Figure~\ref{fig:Nesterov_rate}, we plotted $-1/\log\rho$ versus $L/m$ to get a sense of how the relative iteration count scales as a function of condition number.  As we can see from Figure~\ref{fig:Nesterov_rate}, Nesterov's method applied to quadratics is within a factor of $2$ of the theoretical lower bound, and the bound we can prove for Nesterov's method applied to strongly convex functions is within a factor of $1.4$ of the bound for quadratics.

Finally, we must also ensure that $P$ is reasonably well-conditioned. In Figure~\ref{fig:Nesterov_cond}, we see that $\cond(P)$ appears to be proportional to $L/m$, which agrees with the scale factor found by Nesterov~\cite{NesterovBook}.

If we repeat the above experiments, but instead using the optimal tuning of Nesterov's method given in Proposition~\ref{prop:quadratic-rates}, the resulting plots are virtually identical. The only differences are that the curves are shifted down slightly because the optimal rate for quadratics is now $1-\tfrac{2}{\sqrt{3\kappa+1}}$ instead of $1-\tfrac{1}{\sqrt{\kappa}}$. The sector-IQC curve goes unstable a little sooner as well, at around $L/m \approx 10$. Roughly speaking, if we use the optimal tuning we can guarantee slightly faster convergence but slightly less robustness.
\begin{figure}[ht]
\centering
\includegraphics[width=0.8\linewidth]{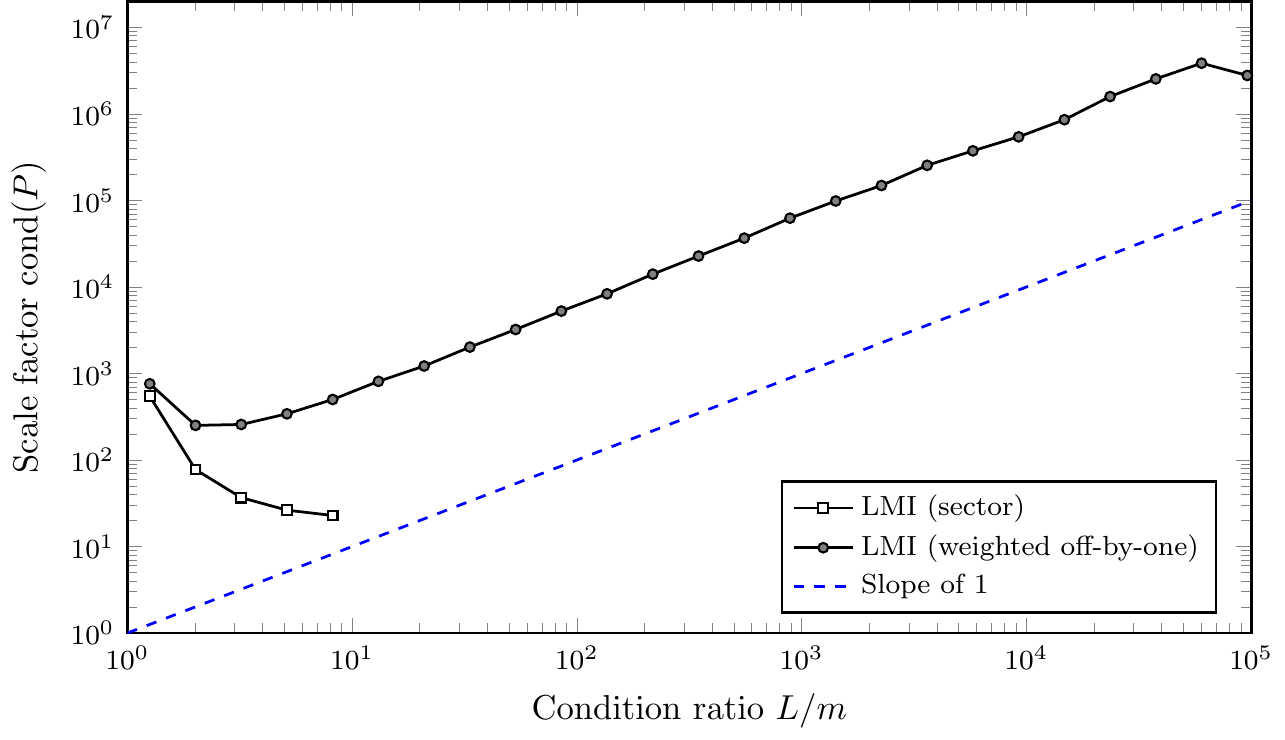}
\caption{Condition number $\cond(P)$ when using the weighted off-by-one IQC. It is within a constant factor of $L/m$. Note that $\log\cond(P)$ appears in~\eqref{dolphin} for computing minimum iterations to convergence.\label{fig:Nesterov_cond}}
\end{figure}

\subsection{The Heavy-ball method}\label{subsec:hb_counterexample}

The optimal Heavy-ball rate for quadratics in Proposition~\ref{prop:quadratic-rates} matches Nesterov's lower bound~\eqref{rho_opt} for strongly convex functions. 
Although the Heavy-ball method and Nesterov's accelerated method have similar recursions, Figures~\ref{fig:Nesterov_rate} and~\ref{fig:Heavyball_rate} tell very different stories.
When we allow for a different $f_k$ at every iteration (sector IQC), we can guarantee stability when $L/m \approx 6$ or less. When we include the weighted off-by-one IQC as well, we can only guarantee stability when $L/m\approx 18$ or less. While it seems possible that using more IQCs could potentially improve this upper bound, it turns out that the poor quality of these bounds is due to something more serious: {\bf the Heavy-ball method optimized for quadratics does not converge for general $f\in\cvx$}.
\begin{figure}[ht]
\centering
\includegraphics[width=0.8\linewidth]{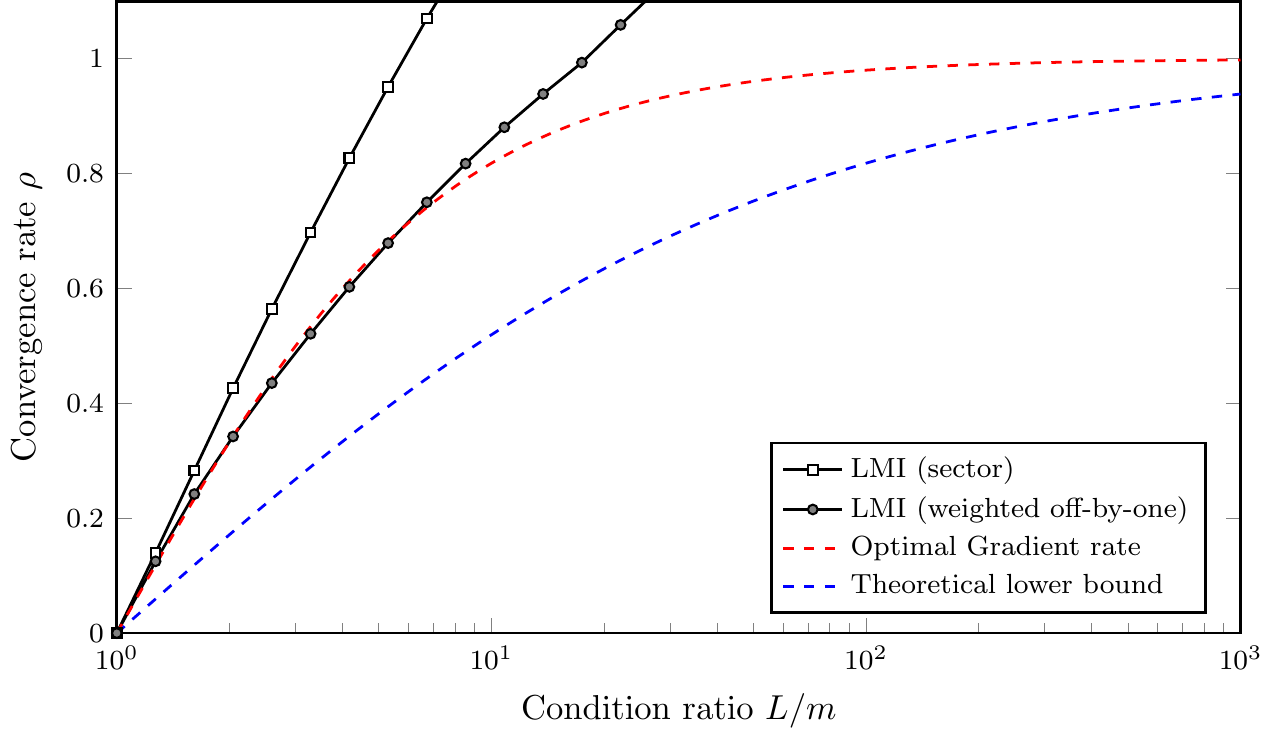}
\includegraphics[width=0.8\linewidth]{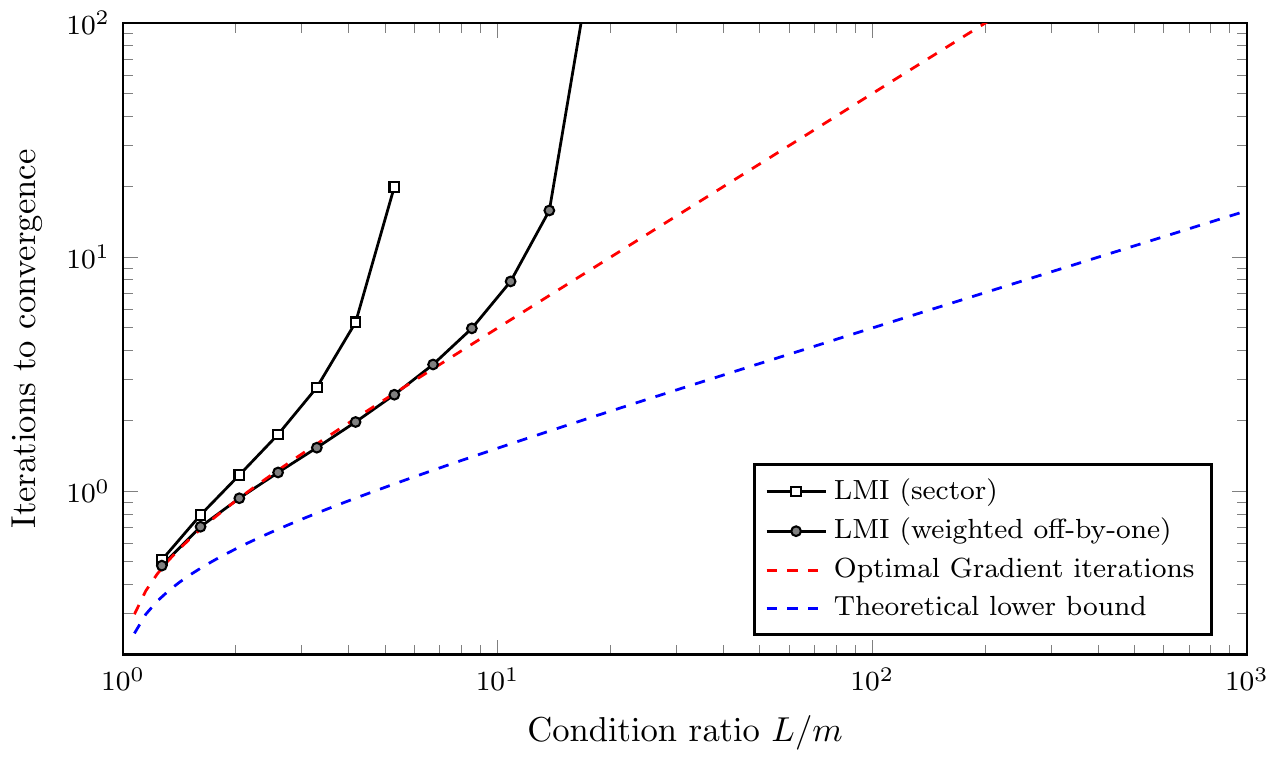}\caption{Upper bounds for the Heavy-ball method, using either the sector IQC or the weighted off-by-one IQC. Convergence rate (first plot) and number of iterations required to achieve convergence to a specified tolerance (second plot). Note that the theoretical lower bound is equal to the optimal Heavy-ball rate for quadratics. The \emph{theoretical lower bound} $\rho_\textup{opt}$ is given in~\eqref{rho_opt}. \label{fig:Heavyball_rate}}
\end{figure}

To find an example of an $f(x)$ that leads to a non-convergent Heavy-ball method, Figure~\ref{fig:Heavyball_rate} indicates that we should search for $L/m > 18$. The following one-dimensional example does the job.
\begin{equation}\label{counterexample}
\grad f(x) = \begin{cases} 25 x & x < 1 \\
					 x + 24 & 1 \le x < 2 \\
					 25 x - 24 & x \ge 2
			\end{cases}
\end{equation}
It is easy to check that $\grad f(x)$ is continuous and monotone, and so $f\in\cvx$ with $m=1$ and $L=25$. When using an initial condition in the interval $3.07 \le x_0 \le 3.46$, the Heavy-ball method produces a limit cycle with oscillations that never damp out. The first 50 iterates for $x_0=3.3$ are shown \color{black} in Figure~\ref{fig:hb_counterexample}, and a plot of $f(x)$ with the limit cycle overlaid is shown in Figure~\ref{fig:hb_limitcycle}.
\begin{figure}[ht]
\centering
\includegraphics[width=0.8\linewidth]{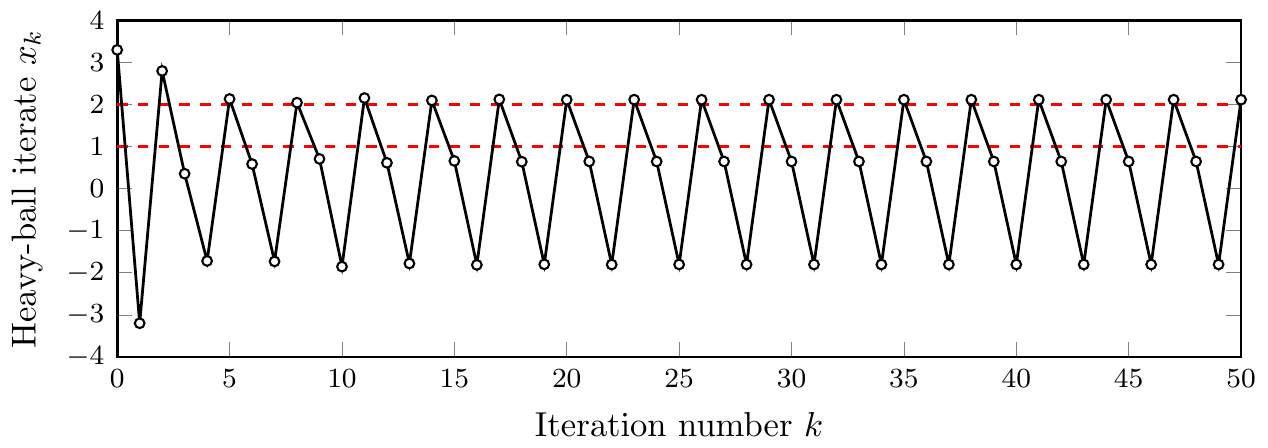}
\caption{Iteration history of the Heavy-ball method when optimizing $f(x)$ defined in~\eqref{counterexample}. Dashed lines separate the pieces of $f(x)$. The iterates tend to a limit cycle, so the Heavy-ball method does not converge for this particular strongly convex function.\label{fig:hb_counterexample}}
\end{figure}

\begin{figure}[ht]
\centering
\includegraphics[width=0.8\linewidth]{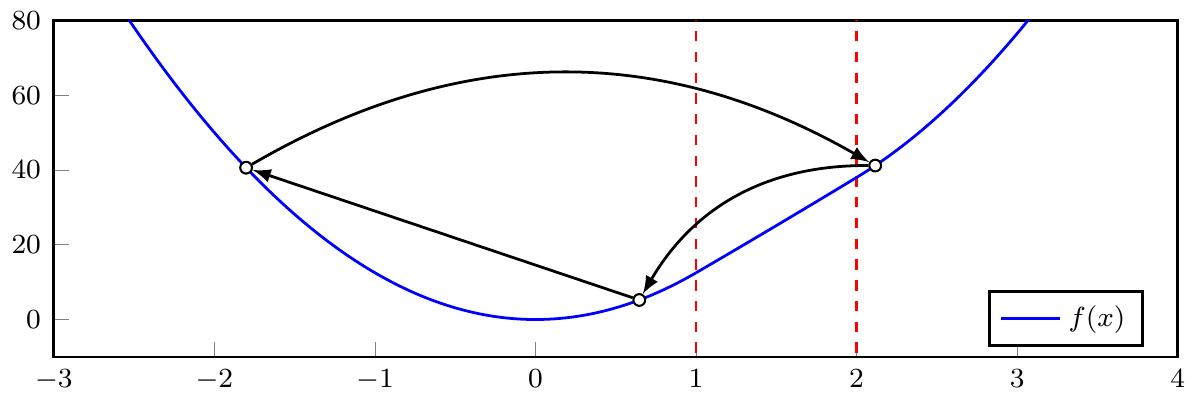}
\caption{Graph of $f(x)$ defined in~\eqref{counterexample} with the limit cycle overlaid on top.\label{fig:hb_limitcycle}}
\end{figure}

For a detailed proof that $f$ can indeed converge to a limit cycle, see Appendix~\ref{appendix:hb_proof}. We further investigate the stability of the Heavy-ball method in Section~\ref{sec:applications}.

%%%%%%%%%%%%%%%%%%%%%%%%%%%%%%%%%%%%%%%%%%%%%%%%%%%%%%%%%%%%%%%%%%%%%%%%%%%%%%%
\section{Further applications}\label{sec:applications}

\subsection{Stability of the Heavy-ball method}
We saw in Section~\ref{subsec:hb_counterexample} that the Heavy-ball method that uses $\alpha$ and $\beta$ optimized for quadratic functions is unstable for general strongly convex functions. A natural question to ask is whether the Heavy-ball method is stable over the class $\cvx$ for \emph{some} choice of $\alpha$ and $\beta$. This experiment is easy to carry out in our framework, because choosing new values of $\alpha$ and $\beta$ simply amounts to changing parameters in the LMI. We chose $\alpha = \tfrac1L$, and for a sampling of points in $\beta\in[0,1]$, we evaluated the corresponding Heavy-ball method using Theorem~\ref{thm:main} together with the weighted off-by-one IQC. See Figure~\ref{fig:Heavyball_grid_rate}.
\begin{figure}[ht]
\centering
\includegraphics[width=0.8\linewidth]{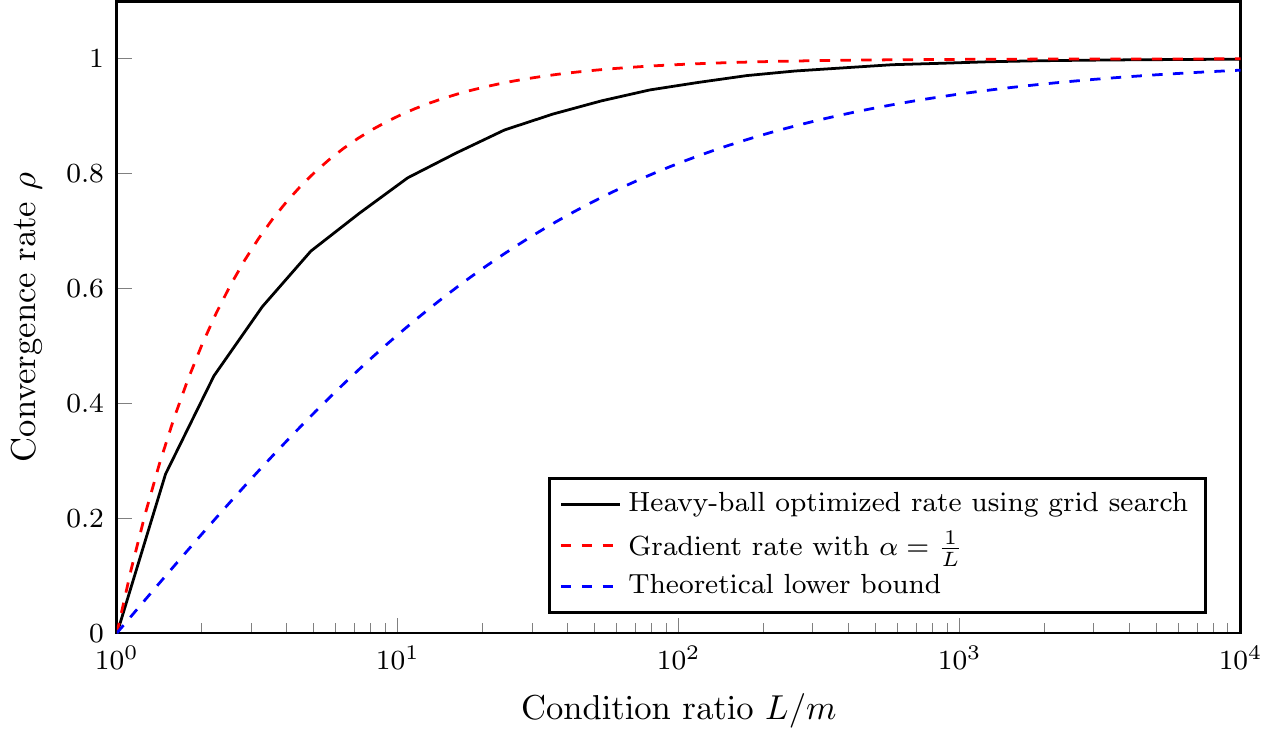}
\includegraphics[width=0.8\linewidth]{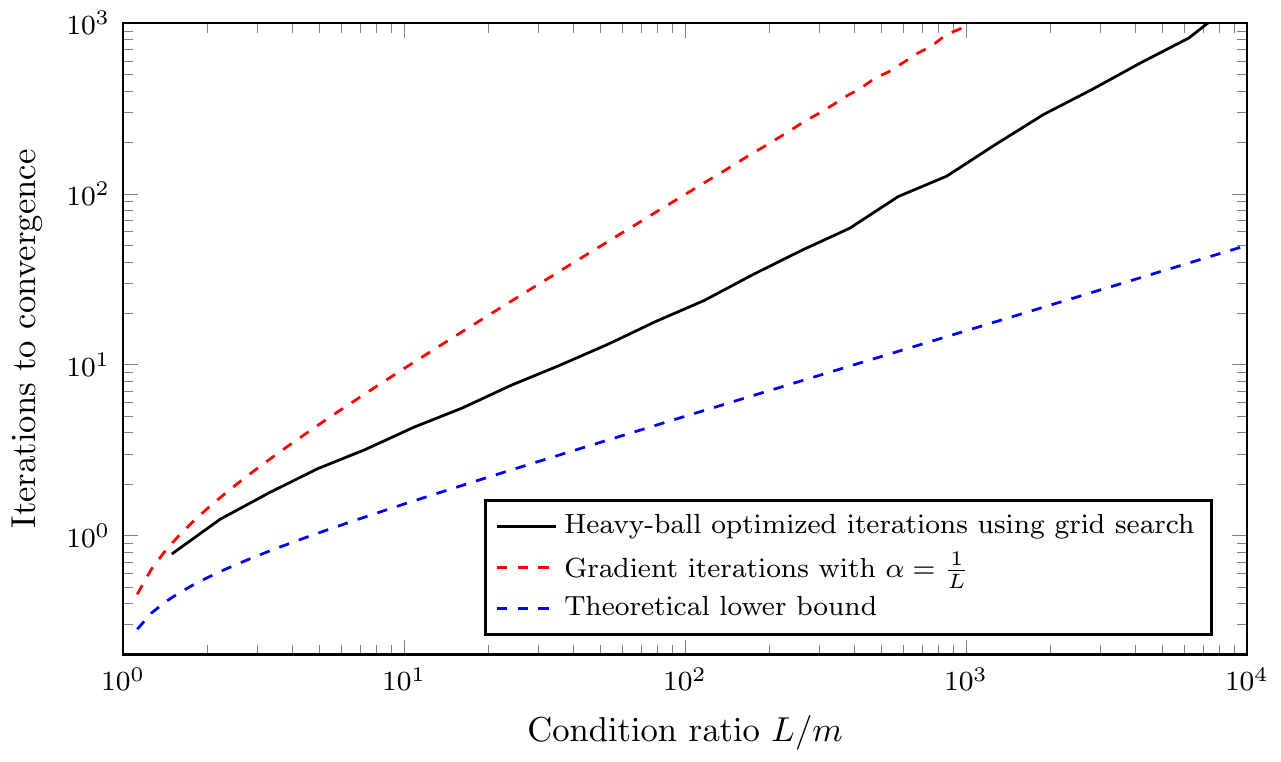}\caption{Upper bounds for the Heavy-ball method. We fixed $\alpha=\tfrac1L$, and for each $L/m$, we picked $\beta$ that led to the optimal rate. The result is the solid black curve. We plotted convergence rate (first plot) and number of iterations required to achieve convergence to a specified tolerance (second plot). The \emph{theoretical lower bound} $\rho_\textup{opt}$ is given in~\eqref{rho_opt} and is the same as the optimal Heavy-ball rate for quadratics.\label{fig:Heavyball_grid_rate}}
\end{figure}

The first plot shows convergence rate. When $\beta=0$, the Heavy-ball method becomes the Gradient method, which is always convergent. However, we can improve upon the gradient rate by optimizing over $\beta$. The best achievable rate is given by the black curve. The black curve lies strictly above the optimal Heavy-ball rate for quadratics, but below the optimal gradient rate.

In the second plot, we show the iterations required to achieve convergence. Again, the black curve represents the optimal parameter choice. As $L/m$ gets large, the envelope veers away from the optimal Heavy-ball curve and becomes parallel to the optimal gradient curve. So when $L/m$ is large, even when $\beta$ is chosen optimally, the Heavy-ball method is comparable to the Gradient method in worst-case for general strongly convex functions.

\subsection{Multiplicative gradient noise}\label{sec:multiplicative_noise}
A common consideration is the inclusion of noise in the gradient computation. One possible model is \emph{relative deterministic noise} where we assume the gradient error is proportional to the distance to optimality~\cite{polyak1987introduction}.  Instead of directly observing $\grad f(y)$, we see
$
	u_k = \grad f(y_k) + r_k\,,
$
where
\[
	\|r_k\| \leq \delta \|\grad f(y_k)\|
\]
for some small nonnegative $\delta$.  The IQC framework can be used to analyze such situations to study the robustness of various algorithms to this type of noise.

If $w_k$ is the true gradient, we actually measure $u_k=\Delta_k w_k$, where the gradient error is bounded above by a quantity proportional to the true gradient. In other words, we assume there is some $\delta > 0$ such that $\|u_k-w_k\| \le \delta \|w_k\|$. Squaring both sides of the inequality and rearranging, we obtain the IQC
\[
\bmat{ w_k \\ u_k }^\tp
\bmat{ \delta^2-1 & 1 \\ 1 & -1 }
\bmat{ w_k \\ u_k } \ge 0
\qquad\text{for all }k
\]
Note that this is simply the sector IQC with $m=1-\delta$ and $L=1+\delta$. We make no assumptions on how the noise is generated; it may be the output of a stochastic process, or could even be chosen adversarially. The modified block-diagram is shown in Figure~\ref{fig:delta}.
\begin{figure}[ht]
\centering
\includegraphics{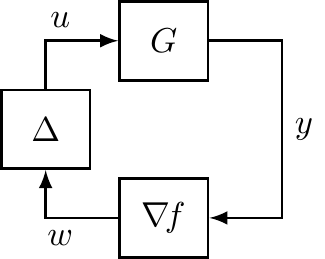}
\caption{Block-diagram representation of the standard interconnection with an additional block $\Delta$ representing multiplicative noise.\label{fig:delta}}
\end{figure}

By making a small modification, we can apply Theorem~\ref{thm:main}. We will look to show that the following inequality holds over all trajectories
\begin{equation}\label{eq:noiselmi}
x_{k+1}^\tp P x_{k+1} - \rho^2\, x_k^\tp P x_k + \lambda_1 z_k^\tp M z_k
+ \lambda_2 \bmat{ w_k \\ u_k }^\tp
\bmat{ \delta^2-1 & 1 \\ 1 & -1 }
\bmat{ w_k \\ u_k }
\le 0
\end{equation}
for some $\lambda_1,\lambda_2\ge 0$. In order to formulate an LMI that implies a solution to~\eqref{eq:noiselmi}, we use the signal $\bmat{x_k^\tp & u_k^\tp & w_k^\tp }$. Consequently, the matrices $(\hat A, \hat B,\hat C,\hat D)$ from \eqref{eq:comboss}--\eqref{eq:comboss2} now become a map $(w_k,u_k) \mapsto z_k$. This leads to an LMI of the form \eqref{eq:mainLMI2} which is now block-$3\times 3$ instead of the $2\times 2$ LMI of Theorem~\ref{thm:main}. The proof is identical to that of Theorem~\ref{thm:main}.

\paragraph{Gradient method} Our first experiment is to test the Gradient method. We used noise values of $\delta\in\{0.01,0.02,0.05,0.1,0.2,0.5\}$. See Figure~\ref{fig:Gradient_pnoise_rate_1}.

\begin{figure}[ht]
\centering
\includegraphics[width=0.8\linewidth]{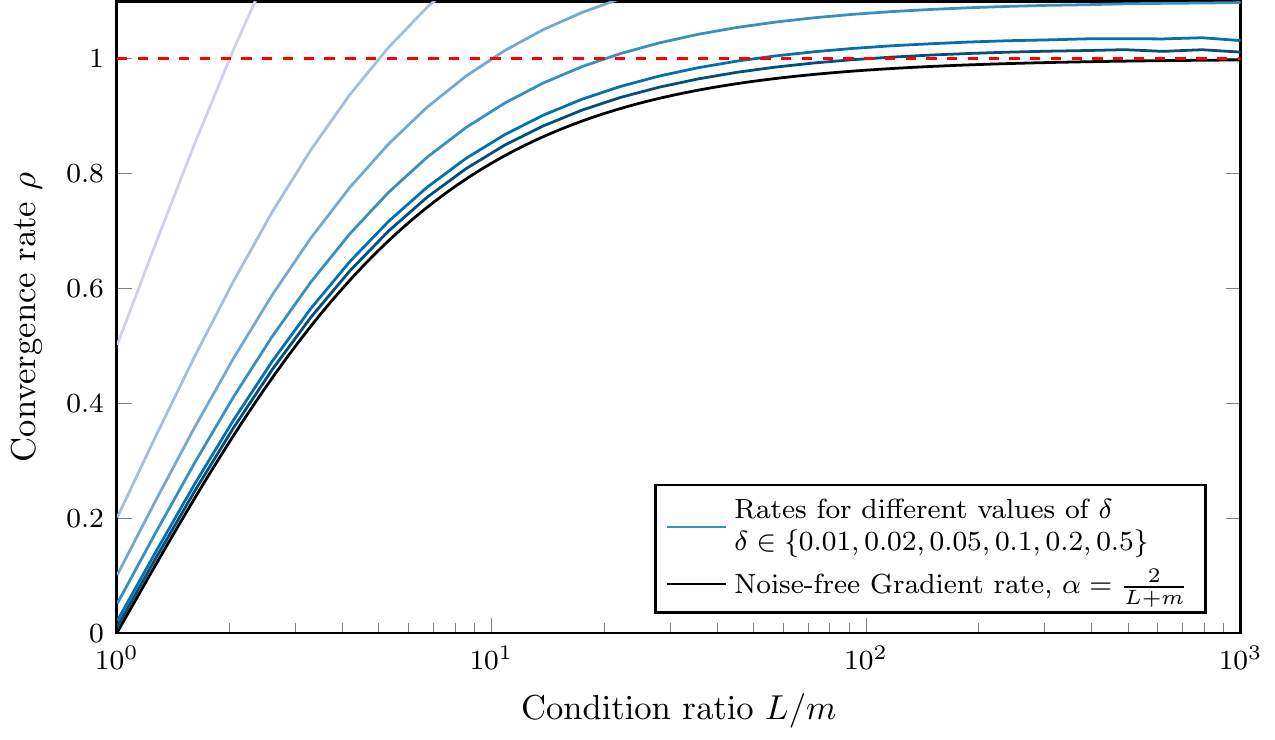}
\includegraphics[width=0.8\linewidth]{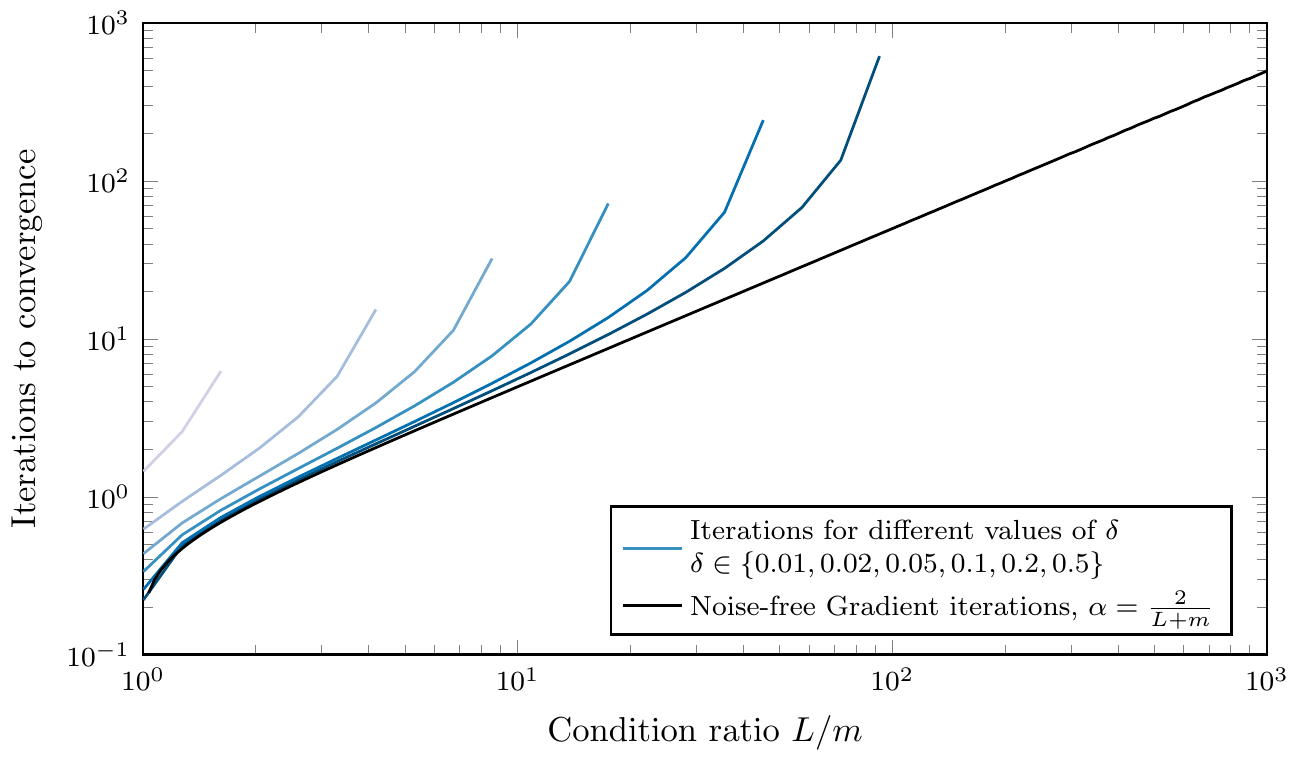}
\caption{Convergence rate and iterations to convergence for the Gradient method with $\alpha=\tfrac{2}{L+m}$, for various noise parameters $\delta$. This method is not robust to noise.\label{fig:Gradient_pnoise_rate_1}}
\end{figure}

In examining Figure~\ref{fig:Gradient_pnoise_rate_1}, we observe that the Gradient method with stepsize $\tfrac{2}{L+m}$ is not very robust to multiplicative noise. Even with noise as low as 1\% ($\delta=0.01$), the Gradient method is no longer stable for $L/m > 100$. An explanation for this phenomenon is that in choosing the stepsize $\alpha$, we are trading off convergence rate with robustness. The choice $\tfrac{2}{L+m}$ yields the minimum worst-case rate, but is fragile to noise. If we pick a more conservative stepsize such as the popular choice $\alpha=\tfrac{1}{L}$, we obtain a very different picture. See Figure~\ref{fig:Gradient_pnoise_rate_2}.
\begin{figure}[ht]
\centering
\includegraphics[width=0.8\linewidth]{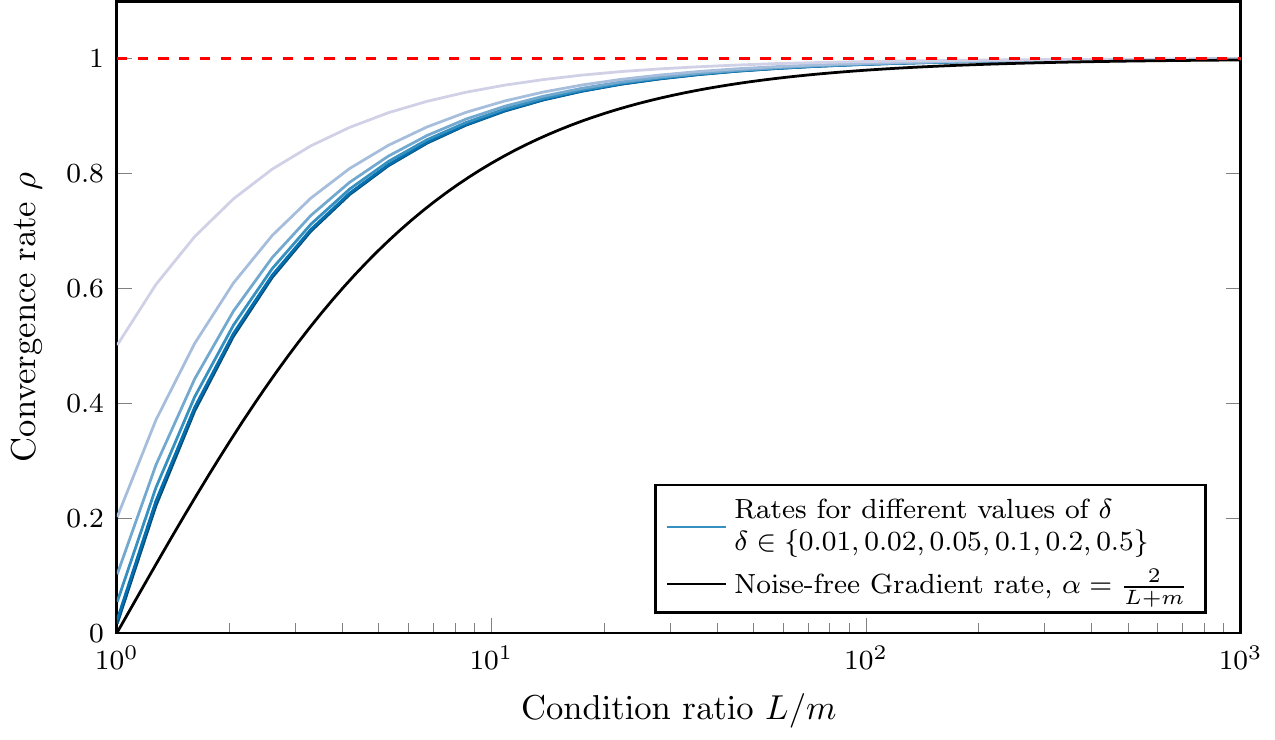}
\includegraphics[width=0.8\linewidth]{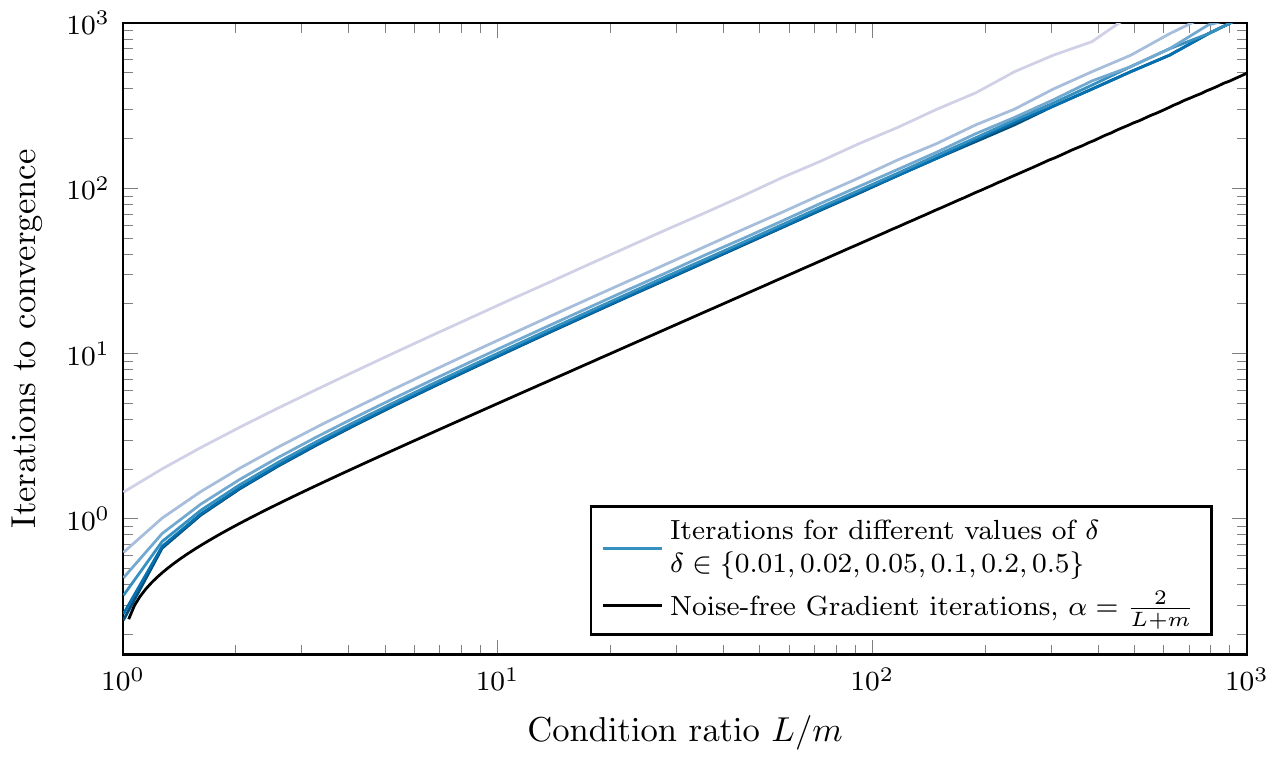}
\caption{
Convergence rate and iterations to convergence for the Gradient method with $\alpha=\tfrac{1}{L}$, for various noise parameters $\delta$. This method is robust to noise, but at the expense of a gap in performance compared to the optimal stepsize of $\alpha=\tfrac{2}{L+m}$.\label{fig:Gradient_pnoise_rate_2}}
\end{figure}

Notice that with the updated stepsize of $\alpha=\tfrac1L$, the Gradient method is now robust to multiplicative noise. Robustness comes at the expense of a degradation in the best achievable convergence rate. This degradation manifests itself as a gap in Figure~\ref{fig:Gradient_pnoise_rate_2} between the black curves and the other ones.

\paragraph{Nesterov's accelerated method} We can carry out an experiment similar to the one we did with the Gradient method, but now with Nesterov's method. As before, we examine the trade-off between the magnitude of the multiplicative noise and the degradation of the optimal convergence rate. This time, we use $\delta\in\{0.05, 0.1,0.2,0.3,0.4,0.5\}$. See Figure~\ref{fig:Nesterov_pnoise}.
\begin{figure}[ht]
\centering
\includegraphics[width=0.8\linewidth]{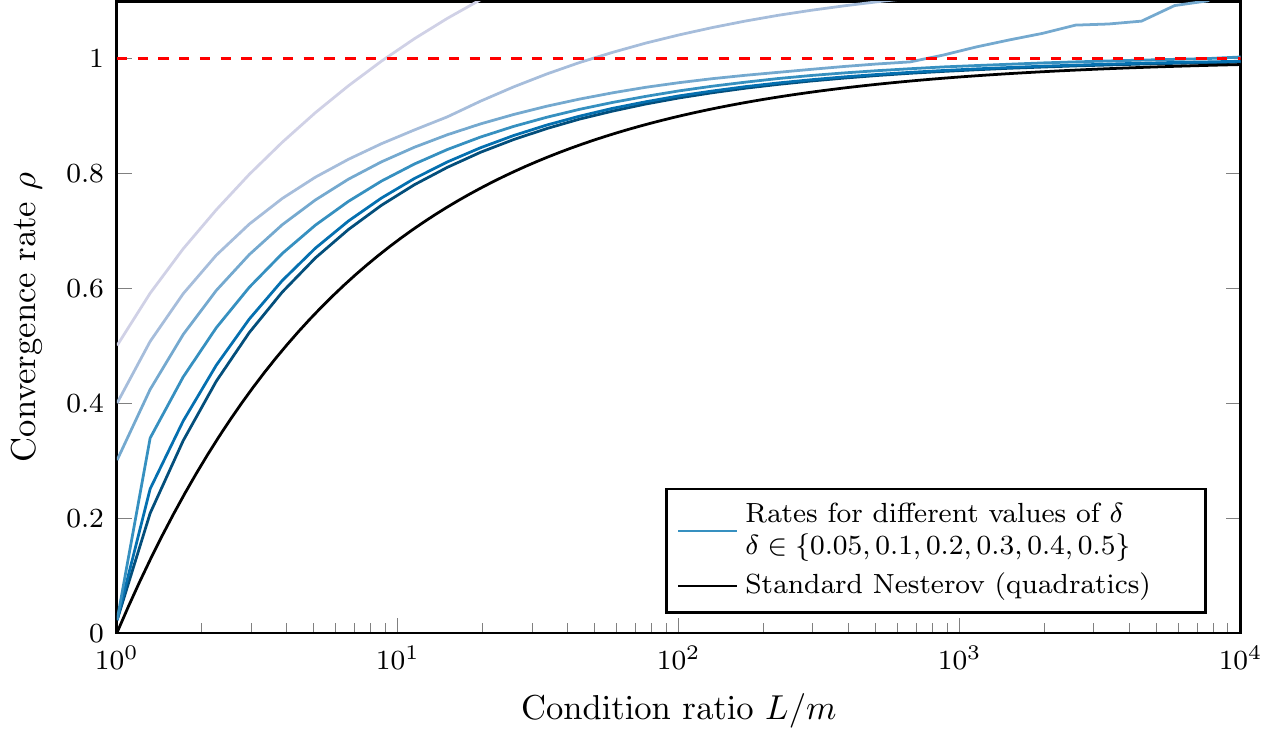}
\includegraphics[width=0.8\linewidth]{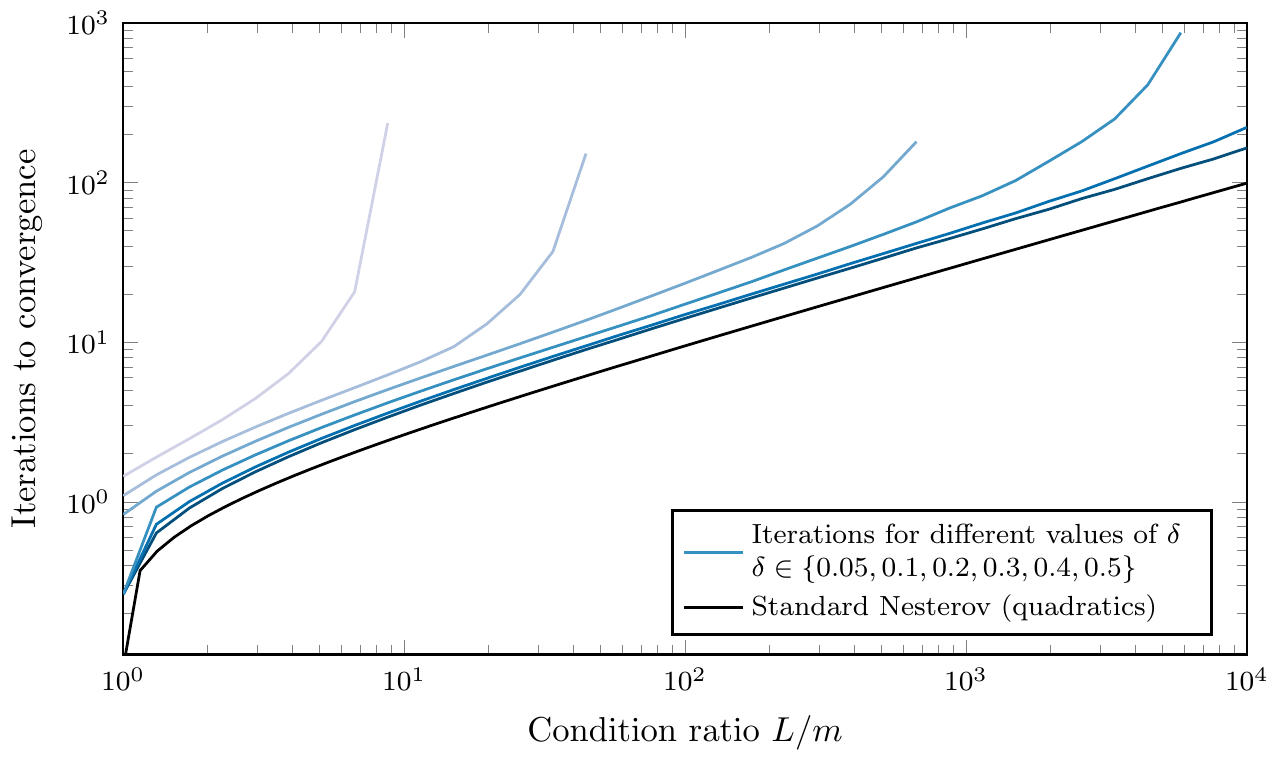}
\caption{
Convergence rate and iterations to convergence for Nesterov's method with standard tuning, for various noise parameters $\delta$.\label{fig:Nesterov_pnoise}}
\end{figure}

As with our first Gradient method test, Nesterov's method is not robust to multiplicative noise. For moderate $L/m$, the degradation is minor, but eventually leads to instability when we reach a certain threshold. The idea that accelerated methods are sensitive to noise and can lead to an accumulation of error was noted in the recent work \cite{devolder2014}, using a different notion of gradient perturbation.

Robustness of Nesterov's method can be improved by modifiying the $\alpha$ and $\beta$ parameters. Choosing a smaller $\alpha$ pushes back the instability threshold, while choosing a smaller $\beta$ simultaneously pushes back the instability threshold and degrades the rate. In the limit $\beta\to 0$, Nesterov's method becomes the Gradient method, so we recover the plots of Figure~\ref{fig:Gradient_pnoise_rate_2}.

\subsection{Proximal point methods}\label{sec:prox-point}

Suppose we are interested in solving a problem of the form
\begin{align*}
\text{minimize}\qquad& f(x) + P(x) \\
\text{subject to}\qquad& x\in\R^n
\end{align*}
where $f\in\cvx$, and $P$ is an extended-real-valued convex function on $\R^n$. An example of such a problem is constrained optimization, where we require that $x\in C$. In this case, we simply let $P$ be the indicator function of $C$. We will now show how the IQC framework can be used to analyze algorithms involving a \emph{proximal} operator. Define the proximal operator of $P$ as
\[
\Pi_{\nu}(x) \defeq \arg\min_y \bl( \tfrac12\|x-y\|^2 + \nu P(y) \br)
\]
As an illustrative example, we will show how to analyze the proximal version of Nesterov's algorithm. Iterations take the form:
\begin{equation}\label{eq:nest_prox}
\begin{aligned}
\xi_{k+1} &= \Pi_\nu\left( y_k - \alpha \grad f(y_k) \right) \\
y_k &=\xi_k + \beta (\xi_k - \xi_{k-1})
\end{aligned}
\end{equation}
Note that when $\Pi_\nu = I$, we recover the standard Nesterov algorithm.  When $\beta=0$, we recover the proximal gradient method. 

In order to analyze this algorithm, we must characterize $\Pi_\nu$ using IQCs. To this end, let $T \defeq \partial P$ be the subdifferential of $P$. Then, $\Pi_\nu(x)$  is the unique point such that $x-\Pi_\nu(x) \in \nu T(\Pi_\nu(x))$. Or, written another way,
\begin{equation}\label{PiT}
\Pi_\nu = (I + \nu T)^{-1}
\end{equation}
Since $T$ is a subdifferential, it satisfies the \emph{incremental passivity} condition. Namely,
\[
(T(x)-T(y))^\tp (x-y) \ge 0
\qquad\text{for all }x,y\in\R^n
\]
Therefore, $T$ satisfies the sector IQC with $m=0$ and $L=\infty$. In fact, via minor modifications of Lemma~\ref{lem:offbyone} and Lemma~\ref{lem:zames-falb} using the definition of a subdifferential rather than~\eqref{eq:coercive2}, $T$ satisfies the off-by-one and weighted off-by-one IQCs as well. Now transform~\eqref{eq:nest_prox} by introducing the auxiliary signals
$u_k\defeq \grad f(y_k)$, $w_k\defeq \Pi_\nu(y_k-\alpha u_k)$, $v_k\defeq \nu T(w_k)$. 
The definitions of $w_k$ and $v_k$ together with~\eqref{PiT} immediately imply that $w_k = y_k - \alpha u_k - v_k$. Therefore, we can rewrite~\eqref{eq:nest_prox} as
\begin{align*}
\begin{aligned}
\xi_{k+1} &= \xi_k + \beta(\xi_k-\xi_{k-1}) - v_k - \alpha u_k \\
w_k &= \xi_k + \beta (\xi_k - \xi_{k-1}) - v_k - \alpha u_k \\
y_k &= \xi_k + \beta (\xi_k - \xi_{k-1})
\end{aligned}
& & \text{with:} & &
\begin{aligned}
u_k &= \grad f(y_k) \\
v_k &= \nu T(w_k)
\end{aligned}
\end{align*}
These equations may be succinctly represented as a block diagram, as in Figure~\ref{fig:pi}.
\begin{figure}[ht]
\centering
\includegraphics{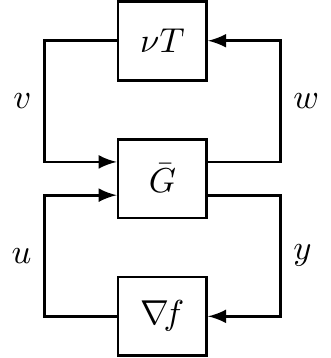}
\caption{Block-diagram representation of the standard interconnection with an additional block $\nu T$ representing a scaled subdifferential.\label{fig:pi}}
\end{figure}

Analyzing this interconnection is done by accounting for the IQCs for both unknown blocks. If $(\Psi_1,M_1)$ is the IQC for $\grad f$ with output $z^1_k$ and $(\Psi_2,M_2)$ is the IQC for $\nu T$ with output $z^2_k$, then we seek to show that for all trajectories satisfy
\begin{equation}\label{eq:projlmi}
x_{k+1}^\tp P x_{k+1} - \rho\, x_k^\tp P x_k + \lambda_1 (z_k^1)^\tp M_1 (z_k^1)
+ \lambda_2 (z_k^2)^\tp M_2 (z_k^2)
\le 0
\end{equation}
where $x_k$ now includes the states $\xi_k$ as well as the internal states of $\Psi_1$ and $\Psi_2$. As in the proof of Theorem~\ref{thm:main}, for each fixed $\rho$, we can write~\eqref{eq:projlmi} as an LMI in the variables $P\succ 0$, $\lambda_1\ge 0$, $\lambda_2\ge 0$.

Applying this approach to the proximal version of Nesterov's accelerated method, we recover the exact same plots as in Figure~\ref{fig:Nesterov_rate}. This is to be expected because it is known that the proximal gradient and accelerated methods achieves the same worst-case convergence rates as their unconstrained counterparts~\cite{beck2009fast,NesterovComposite,tseng2008accelerated}.  We conjecture that any algorithm $G$ of the form~\eqref{eq:lure} which converges with rate $\rho$ has a proximal variant that converges at precisely the same rate.

\subsection{Weakly convex functions}

With minor modifications to our analysis, we can immediately extend our results to the case where the function to be optimized is convex, but not strongly convex. Specifically, we will assume throughout this subsection that $f\in\cvxw$. The following development is due to Elad Hazan~\cite{HazanPersonal}.

Suppose we want to minimize $f$ over a compact, convex domain $\mathcal{D}$ for which we can readily compute the Euclidean projection.  Let $R$ denote the diameter of the set $\mathcal{D}$. Define the function $f_\epsilon(x) \defeq f(x)+\tfrac{\epsilon}{2R^2} \|x\|^2$. Note that $f_\epsilon$ is differentiable and strongly convex; it satisfies $f_\epsilon\in S(\tfrac{\epsilon}{R^2},L+\tfrac{\epsilon}{R^2})$. Therefore, we may apply our analysis to $f_\epsilon$.

Suppose we execute on $f_\epsilon$ an algorithm with interleaved projections as in Section~\ref{sec:prox-point}.  Let $x_\star$ be any minimizer of $f$ on $\mathcal{D}$ and $x_\star^{(\epsilon)}$ be the minimizer of $f_\epsilon$.   Let $\rho$ denote the rate of convergence achieved when the condition ratio is set as $\kappa = (1+LR^2/\epsilon)$ and let $P_\epsilon \succ 0$ be the associated solution to the LMI.  Let $\sigma \defeq \cond(P_\epsilon)$. After $k$ steps,
\begin{align*}
	f(x_k) - f(x_\star) &= f_\epsilon(x_k) - f_\epsilon(x_\star) + \frac{\epsilon}{2R^2} \left(\|x_\star\|^2 - \|x_k\|^2\right)\\
	 &\leq f_\epsilon(x_k) - f_\epsilon(x_\star^{(\epsilon)}) + \frac{\epsilon}{2R^2} \left(\|x_\star\|^2 - \|x_k\|^2\right)\\
	 &\leq f_\epsilon(x_k) - f_\epsilon(x_\star^{(\epsilon)}) + \frac{\epsilon}{2}
\end{align*}
Now apply~\eqref{eq:basicdef} from Proposition~\ref{prop:basic_properties} using $(f,x,y) = (f_\epsilon,x_\star^{(\epsilon)},x_k)$ and obtain
\begin{align*}
f(x_k) - f(x_\star)
&\leq  \frac{LR^2+\epsilon}{2R^2} \|x_k - x_\star^{(\epsilon)} \|^2 + \frac{\epsilon}{2}\\
&\leq \frac{LR^2+\epsilon}{2R^2}\, \sigma  \rho^{2k} \|x_0 - x_\star^{(\epsilon)}\|^2 + \frac{\epsilon}{2}\\
&\leq \frac12 \left( (LR^2+\epsilon)\, \sigma  \rho^{2k} + \epsilon \right)\,.
\end{align*}
Where the last inequality follows from the definition of set diameter. Therefore, if 
\begin{equation}\label{porpoise}
k \geq \frac{\log \bl( (1+LR^2/\epsilon)\, \sigma  \br)}{2\log ( \rho^{-1})}\,,
\end{equation}
then $f(x_k)-f(x_\star) \leq \epsilon$. Substituting the rates found algebraically for the quadratic case (Section~\ref{sec:quadratic_case}) or our numerical results for the strongly convex case (Sections~\ref{grad}--\ref{nest}), the convergence rate $\rho$ satisfies
\begin{align*}
\frac{1}{\log(\rho^{-1})} &\propto \kappa = (1+LR^2/\epsilon) & &\text{for the Gradient method, and} \\
\frac{1}{\log(\rho^{-1})} &\propto \kappa^{1/2} = (1+LR^2/\epsilon)^{1/2} & &\text{for Nesterov's accelerated method.}
\end{align*}
Finally, note that $\sigma = \cond(P_\epsilon)$ also depends on $\epsilon$. We can control the growth of $\sigma$ directly by including a constraint of the form $I \preceq P \preceq \sigma I$ when solving the SDP of Theorem~\ref{thm:main}. Alternatively, we can observe (see Figure~\ref{fig:Nesterov_cond}) that $\sigma \propto \kappa = (1+LR^2/\epsilon)$. Therefore, we conclude that
\begin{align*}
k &= \mathcal{O}(\tfrac{1}{\epsilon}\log{\tfrac{1}{\epsilon}})
&& \text{for the Gradient method, and} \\
k &= \mathcal{O}(\tfrac{1}{\sqrt{\epsilon}}\log{\tfrac{1}{\epsilon}})
&& \text{for Nesterov's accelerated method.}
\end{align*}
This analysis matches the standard bounds up to the logarithmic terms~\cite{NesterovBook}.

\section{Algorithm design}

In this section, we show one way in which the IQC analysis framework can be used for algorithm design. We saw in Section~\ref{sec:multiplicative_noise} that the Gradient method can be very robust to noise (Figure~\ref{fig:Gradient_pnoise_rate_2}), or not robust at all (Figure~\ref{fig:Gradient_pnoise_rate_1}), depending on whether we use a stepsize of $\alpha=1/L$ or $\alpha=2/(L+m)$, respectively.

A natural question to ask is whether such a trade-off between performance and robustness exists with Nesterov's method as well. As can be seen in Figure~\ref{fig:Nesterov_pnoise}, Nesterov's method is only somewhat robust to noise. In the sequel, we will synthesize variants of Nesterov's method that explore the performance-robustness trade-off space.

Consider an algorithm of the form~\eqref{eq:lure}. Based on the discussion in Section~\ref{sec:howdoweprove}, we know $A$ must have an eigenvalue of $1$. Moreover, given any invertible $T$, the algorithms $(A,B,C,D)$ and $(TAT^{-1},TB,CT^{-1},D)$ are \emph{equivalent realizations} in the sense that if one is stable with rate $\rho$, the other is stable with rate $\rho$ as well. Indeed, if the first algorithm has state $\xi_k$, the second algorithm has state $T\xi_k$. We limit our search to the case $A\in\R^{2\times 2}$ and $D=0$. Three parameters are required to characterize all possible algorithms in this family (modulo equivalences due to a choice of $T$). One possible parameterization is given by
\begin{equation}\label{eq:newhotness}
\stsp{A}{B}{C}{D} = 
\left[\begin{array}{cc|c}
\beta_1+1 & -\beta_1 & -\alpha \\
1 & 0 & 0 \\ \hlinet
\beta_2+1 & -\beta_2 & 0
\end{array}\right]
\qquad \text{with:}\quad(\alpha,\beta_1,\beta_2) \in \R^3
\end{equation}
In light of the discussion in Section~\ref{sec:opt-dyn}, we see that the Gradient, Heavy-ball, and Nesterov methods are all special cases of~\eqref{eq:newhotness}. In particular,
\[
(\alpha,\beta_1,\beta_2) \text{ is equal to: }
\begin{cases}
(\alpha,0,0) & \text{for the Gradient method} \\
(\alpha,\beta,0) & \text{for the Heavy-ball method} \\
(\alpha,\beta,\beta) & \text{for Nesterov's method}
\end{cases}
\]
We may also rewrite~\eqref{eq:newhotness} in more familiar recursion form as
\begin{subequations}\label{eq:newhotness2}
\begin{align}
\xi_{k+1} &= \xi_{k} - \alpha\grad f(y_k) + \beta_1(\xi_k - \xi_{k-1}) \\
y_{k} &= \xi_k + \beta_2(\xi_k - \xi_{k-1})
\end{align}
\end{subequations}
Our approach is straightforward: for each choice of condition ratio $L/m$ and noise strength~$\delta$, we generate a large grid of tuples $(\alpha,\beta_1,\beta_2)$ and use the approach of Section~\ref{sec:multiplicative_noise} to evaluate each algorithm. We then choose the algorithm with the lowest~$\rho$. In other words, given bounds on the condition ratio and noise strength, we choose the algorithm for which we can certify the best possible convergence rate over all admissible choices of $f$ and gradient noise. The performance of each optimized algorithm is plotted in Figure~\ref{fig:brute_force}.

\begin{figure}[ht]
\centering
\includegraphics[width=.8\linewidth]{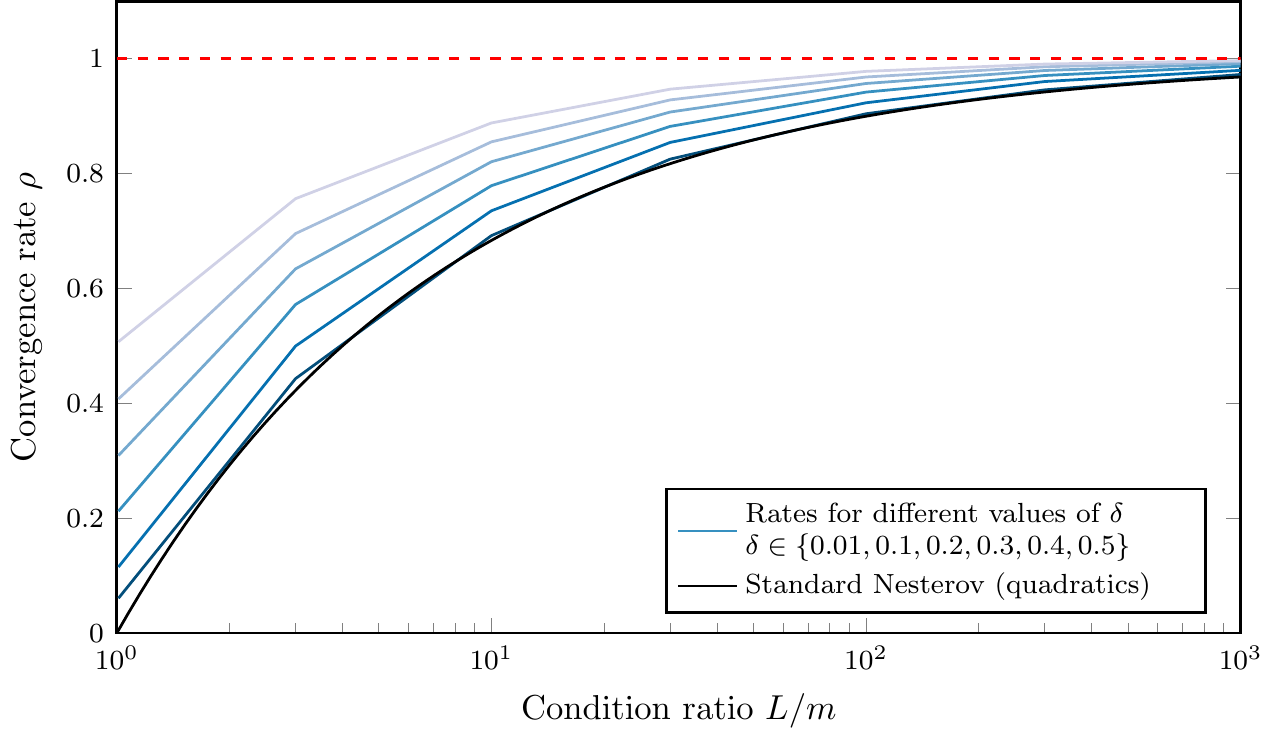} 
\includegraphics[width=.8\linewidth]{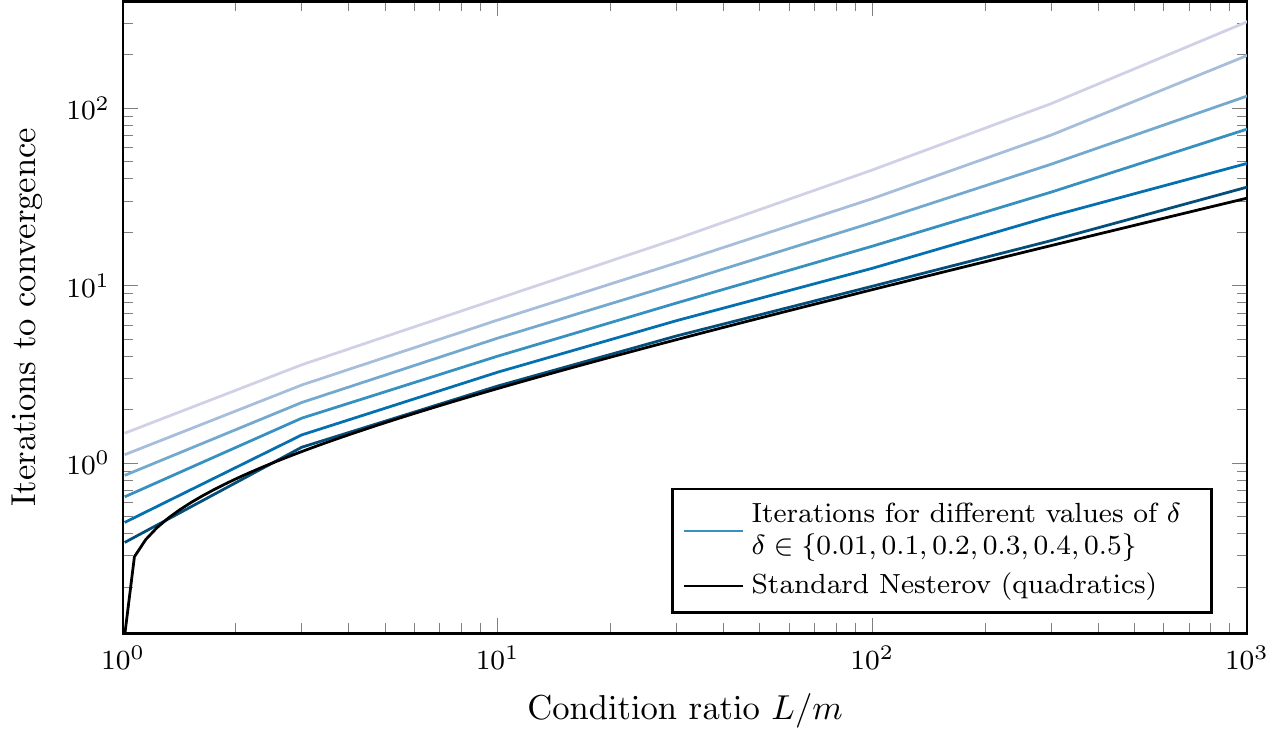}
\caption{Upper bounds found using a brute-force search over the three-parameter family of algorithms described by~\eqref{eq:newhotness2}. Convergence rate is shown (first plot) as is the number of iterations required to achieve convergence to a specified tolerance (second plot). Although the bounds assume strongly convex functions, we also show the worst-case rate for quadratics as a comparison.\label{fig:brute_force}}
\end{figure}

By design, this new family of algorithms must have a performance superior to the Gradient method, Nesterov's method, and the Heavy-ball method for any choice of tuning parameters.
In the limit $\delta\to 0$, we appear to recover the performance of Nesterov's method when it is applied to \emph{quadratics}.  That is, we have used numerical search to find an algorithm whose worst case performance guarantee is slightly better than what is guaranteed by Nesterov's method.

In the second plot of Figure~\ref{fig:brute_force}, the algorithms robust to higher noise levels have greater slopes. When the noise level is low ($\delta=0.01$), we approach a slope of 0.5, the same as Nesterov. When the noise level is high ($\delta=0.5$), the slope is roughly 0.75. Note that the Gradient method, which was robust for \emph{all} noise levels, has a slope of 1. Therefore, the new algorithms we found explore the trade-off between noise robustness and performance, and may be useful in instances where Nesterov's method would be too fragile and the Gradient method would be too slow.

\section{Future work}

We are only beginning to get a sense of what IQCs can tell us about optimization schemes, and there are many more control theory tools and techniques left to adapt to the context of optimization and machine learning. We conclude this paper with several interesting directions for future work.

\paragraph{Analytic proofs}  One of the drawbacks of our numerical proofs is that we are always pushing up against numerical error and conditioning error.  Analytic proofs would alleviate this issue and could provide more interpretable results about how parameters of algorithms should vary to meet performance and robustness demands.  To provide such analytic proofs, one would have to solve small LMIs in closed form.  This amounts to solving small semidefinite programming problems, and this may be doable using analytic tools from algebraic geometry~\cite{grayson2002macaulay, rostalski2010dualities}. 

\paragraph{Lower Bounds}  Our IQC conditions are merely sufficient for verifying the convergence of an optimization problem.  However, as pointed out by Megretski and Rantzer,  the derived conditions are necessary in a restricted sense~\cite{megrantzer}.  If we fail to find a solution to our LMI, then there is necessarily a sequence of point that satisfy all of the IQC constraints and that do not converge to an equilibrium~\cite{Megretski93,Shamma92}.  It is thus possible that this tool can be used to construct a convex function to serve as a counterexample for convergence. This intuition was what guided our construction of a counterexample for the convergence of the Heavy-ball method.  It may be possible that this construction can be generalized to systematically produce counterexamples.

\paragraph {Time-varying algorithms} In many practical scenarios, we know neither the Lipschitz constant $L$ nor the strong convexity parameter $m$.  Under such conditions, some sort of estimation scheme is used to choose the appropriate step size.  This could be a simple backoff scheme to ensure a sufficient decrease, or a more intricate search method to find the appropriate parameters~\cite{NocedalWrightBook}.  From our control vantage point, it may be possible to use techniques from adaptive control to certify when such line search methods are stable.  In particular, these could be used to differentiate between the different sorts of schemes used to choose the parameters of the nonlinear conjugate gradient method. Useful connections are made between robustness analysis of adaptive controllers and Lyapunov theory in~\cite{sastry2011adaptive}.

A related area of study is that of linear parameter varying (LPV) systems. This extension of linear systems analysis considers parameterized variations in the dynamical system matrices $(A,B,C,D)$. Algorithms with variable stepsize are examples of LPV systems. Some recent work discussing IQCs applied to LPV systems appeared in~\cite{SeilerLPVIQC}.
Another possible direction would be to use optimal control techniques directly to choose algorithm parameters, possibly solving a small SDP at every iteration to choose new assignments.

\paragraph{Algorithm synthesis} Perhaps even more ambitiously than using our framework for parameter selection, our initial results show that we can use IQCs as a way of designing new algorithms.  We restricted our attention to algorithms with one-step of memory, as then we only had to search over 3 parameters.  However, new techniques would be necessary to explore more complicated algorithms.  Local search heuristics could be used here to probe the feasible region of the associated LMIs, but convex methods and convex relaxations may also be applicable and should be investigated for these searches.  

\paragraph{Noise analysis} Our robustness analysis only allows us to consider certain forms of deterministic noise.  Expanding our techniques to study stochastic noise would expand the applicability of our techniques and could provide new insight into popular stochastic optimization algorithms such as stochastic coordinate descent and stochastic gradient descent~\cite{Nemirovski09,nesterov2012efficiency}.  Many of the most common techniques for proving convergence of stochastic methods rely on Lyapunov-type arguments, and we may be able to generalize this approach to account for the variety of different methods.  In order to expand our techniques to this space, we would need to introduce IQCs that were valid~\emph{in expectation}.  Stability methods from stochastic control may be applicable to such investigations.

\paragraph{Beyond convexity} Since our analysis decouples the derivation of constraints on function classes from the algorithm analysis, it is possible that it can be generalized to nonconvex optimization.  If we can characterize the function class by reasonable quadratic constraints, our framework immediately applies, and may lead to entirely new analyses for nonconvex function classes. For example, IQCs for saturating nonlinearities are readily available in the controls literature~\cite{Kulkarni99,megrantzer}.  From a complementary perspective, if we know that our function is not merely convex, but has additional structure, this can be incorporated as additional IQCs.  With extra constraints, it is possible that we can derive faster rates or more robustness for smaller function classes.

\paragraph{Non-quadratic Lyapunov functions}  There has been substantial work in the past decade on  efficient algorithms to search over \emph{non-quadratic} Lyapunov functions~\cite{papachristodoulou02,Parrilo2003}.  These techniques use sum-of-squares hierarchies to certify that non-quadratic polynomials are nonnegative, and still reduce to solving small semidefinite programming problems.  This more general class of Lyapunov functions could be better matched to certain classes of functions  than quadratics, and we could perhaps analyze more complicated algorithms and interconnections.

\paragraph{Large-scale composite system analysis}  Perhaps the most ambitious goal of this program is to move beyond convex models and attempt to analyze complicated optimization systems used in science and industry.  Powerful modeling languages like AMPL or GAMS allow for local analysis of large, complex systems, and certifying that the decisions about these systems are valid and safe would have impact in a variety of fields including process technology, web-scale analytics, and power management.  Since our methods nicely abstract beyond two interconnected systems, it is our hope that they can be extended to analyze the variety of optimization algorithms deployed to handle large, high throughput data processing.

%%%%%%%%%%%%%%%%%%%%%%%%%%%%%%%%%%%%%%%%%%%%%%%%%%%%%%%%%%%%%%%%%%%%%%%%%%%%%%%

\section*{Acknowledgments}
We would like to thank Peter Seiler for many helpful pointers on time-domain IQCs, Elad Hazan for his suggestion of how to analyze functions that are not strongly convex, and Bin Hu for pointing out a misreading of Nesterov's results in an earlier draft of this paper.  We would also like to thank Ali Jadbabaie, Pablo Parrilo, and Stephen Wright for many helpful discussions and suggestions.

LL and AP are partially supported by AFOSR award FA9550-12-1-0339 and NASA  Grant No.~NRA~NNX12AM55A. BR is generously supported by ONR awards N00014-11-1-0723 and N00014-13-1-0129, NSF award CCF-1148243, AFOSR award FA9550-13-1-0138, and a Sloan Research Fellowship.  This research was also supported in part by NSF CISE Expeditions Award CCF-1139158, LBNL Award 7076018, and DARPA XData Award FA8750-12-2-0331, and gifts from Amazon Web Services, Google, SAP, The Thomas and Stacey Siebel Foundation, Adobe, Apple, Inc., C3Energy, Cisco, Cloudera, EMC, Ericsson, Facebook, GameOnTalis, Guavus, HP, Huawei, Intel, Microsoft, NetApp, Pivotal, Splunk, Virdata, Fanuc, VMware, and Yahoo!.

\if\MODE2
\bibliographystyle{siam}
\else
\bibliographystyle{abbrv}
\fi

\bibliography{iqcopt}

\appendix

\section{Proof of Proposition~\ref{prop:quadratic-rates}}\label{A:prop1proof}

\noindent Suppose $Q$ has eigenvalues that satisfy
$0 < m \leq \lambda_d \le \lambda_{d-1} \le \cdots \le \lambda_2 \le \lambda_1 \leq L$.
Throughout, we assume $(A,B,C)$ are the state transition matrices of the algorithm we would like to analyze (see Section~\ref{sec:opt-dyn}). The state transition matrices are functions of the algorithm parameters (e.g. $\alpha$ and  $\beta$ for the Heavy-ball method). Let $T$ be the closed loop system $T \defeq A+BQC$. The worst-case convergence rate is found by maximizing the spectral radius over all admissible $Q$. In other words,
\begin{align*}
\rho_\textup{worst} = \maximize_{mI_d \preceq Q \preceq LI_d} \rho(T)
\end{align*}
The first observation is that for our algorithms of interest, we may assume $d=1$. To see why, take for example the Heavy-ball method, where
\[
T = \bmat{ (1+\beta)I_d - \alpha Q & -\beta I_d \\ I_d & 0_d }
\]
Write the eigenvalue decomposition of $Q$ as $Q= U \Lambda U^\tp$, where $\Lambda = \diag (\lambda_1,\lambda_2, \dotsc, \lambda_d)$ and $U$ is orthogonal. Then,
\[
T = \bmat{U & 0_d \\ 0_d & U}\bmat{ (1+\beta)I_d - \alpha \Lambda & -\beta I_d \\ I_d & 0_d }\bmat{U & 0_d \\ 0_d & U}^\tp
\]
Therefore, by similarity, the eigenvalues of $T$ are the eigenvalues of all the matrices
\[
T_i = \bmat{ (1+\beta) - \alpha \lambda_i & -\beta \\ 1 & 0 },
\quad i=1,2,\dots,d.
\]
and we may without loss of generality let $d=1$. The simplified problem is therefore
\[
\rho_\textup{worst} = \maximize_{m \le \lambda \le L} \rho(T_0)
\]
where $T_0$ is defined as
\[
T_0 = \begin{cases}
1 - \alpha \lambda & \text{Gradient method} \\
\bmat{ (1+\beta)(1-\alpha \lambda) & -\beta(1-\alpha\lambda) \\ 1 & 0 } & \text{Nesterov's method} \\[4mm]
\bmat{ 1+\beta - \alpha \lambda & -\beta \\ 1 & 0 } & \text{Heavy-ball method}
\end{cases}
\]
It is now a matter of algebraic substitution to find the optimal rates for each parameter choice. For example, with the Gradient method,
\begin{align*}
\rho_\textup{max} &= \maximize_{m \le \lambda \le L} \rho(1-\alpha \lambda) \\
&= \max \bbl\{ |1-\alpha m|, |1-\alpha L| \bbr\}
\end{align*}
The second equality follows from the fact that the maximum of a convex function must occur at the boundary. We can now see that when $\alpha = 1/L$, we have $\rho_\textup{max} = 1-1/\kappa$. Finding the optimal $\alpha$ is straightforward in this case because the pointwise maximum of convex functions is itself convex. In this case, the minimum $\rho_\textup{max}$ occurs when $\alpha = \tfrac{2}{L+m}$ and the result is $\rho_\textup{max} = \tfrac{\kappa-1}{\kappa+1}$.

The analyses for Nesterov's method and the Heavy-ball method are similar in spirit to that of the Gradient method, but computing the spectral radius is more complicated. For Nesterov's method, we have
\[
\rho_\textup{max} = \maximize_{m \le \lambda \le L} \,\max\bl\{ |\nu_1|, |\nu_2| \br\}
\]
where $\nu_1$, $\nu_2$ are the roots of the characteristic polynomial of $T_0$, which is
\[
\nu^2 - (1+\beta)(1-\alpha\lambda)\nu + \beta(1-\alpha\lambda) = 0
\]
The magnitudes of the roots satisfy:
\[
\max\bl\{ |\nu_1|, |\nu_2| \br\} = \begin{cases}
\tfrac{1}{2}\big|(1+\beta)(1-\alpha\lambda)\big| + \tfrac{1}{2}\sqrt{\Delta} & \text{if } \Delta \ge 0 \\
\sqrt{\beta(1-\alpha\lambda)} 	& \text{otherwise}
\end{cases}
\]
where $\Delta \defeq (1+\beta)^2(1-\alpha\lambda)^2 - 4\beta(1-\alpha\lambda)$.
It is straightforward to verify that if $\alpha,\beta$ are fixed, $\max\bl\{ |\nu_1|, |\nu_2| \br\}$ is a continuous and quasiconvex function of $\lambda$. So, the maximum over $\lambda$ must occur at boundary point. For the case where $\alpha=\tfrac{1}{L}$ and $\beta = \tfrac{\sqrt{\kappa}-1}{\sqrt{\kappa}+1}$, choosing $\lambda=L$ yields zero, so the maximum  must be achieved at $\lambda=m$, which yields
\[
\rho_\textup{max} = \sqrt{\frac{\sqrt{\kappa}-1}{\sqrt{\kappa}+1}\left(1-\frac{1}{\kappa}\right)}
= \sqrt{\frac{\sqrt{\kappa}-1}{\sqrt{\kappa}+1} \cdot \frac{(\sqrt{\kappa}+1)(\sqrt{\kappa}-1)}{\kappa}}
= 1 - \frac{1}{\sqrt{\kappa}},
\]
as required. When optimizing over quadratic functions, the above tuning of Nesterov's method is suboptimal. Finding the choice of $\alpha$ and $\beta$ that yields the smallest $\rho_\textup{max}$  requires careful examination of several subcases, and we omit the details in the interest of space. The result is shown in the second last row of the table in Proposition~\ref{prop:quadratic-rates}.

A similar eigenvalue analysis was used in \cite{polyak1987introduction} to derive the optimal parameter tuning for the Heavy-ball method applied to quadratics (last row of the table).

\section{Proof of the Heavy-ball counterexample}\label{appendix:hb_proof}

We would like to minimize the function whose gradient is given by~\eqref{counterexample}. The Heavy-ball method with $L=25$ and $m=1$ is given in Section~\ref{sec:quadratic_case}:
\begin{align}\label{hbsp}
x_{k+1} &= \tfrac{13}{9}x_k - \tfrac{4}{9} x_{k-1} -\tfrac{1}{9}\grad f (x_k)
\end{align}
where we use the initialization $x_{-1}=x_0$. Based on the plot of Figure~\ref{fig:hb_counterexample}, we will look for limit points $p,q,r$ such that we have a cycle of period 3:
\begin{equation}\label{hhhhh}
\begin{aligned}
x_{3n} &\to p,  &
x_{3n+1} &\to q, &
x_{3n+2} &\to r
\end{aligned}
\qquad\text{for }n=0,1,\dots
\end{equation}
where $p<1$, $q<1$, and $r > 2$. Substituting the forms~\eqref{hhhhh} and~\eqref{counterexample} directly into~\eqref{hbsp}, we obtain the system of linear equations 
\[
\bmat{ 4 & 12 & 9 \\ 9 & 4 & 12\\ 12 & 9 & 4 }\bmat{p \\ q \\ r} = \bmat{0 \\ 24 \\ 0}
\]
and the unique solution of these equations is
\begin{equation}\label{asdf}
p=\frac{792}{1225} \approx 0.65,\quad
q=-\frac{2208}{1225} \approx -1.80,\quad
r=\frac{2592}{1225} \approx 2.12
\end{equation}
In other words, the trajectory~\eqref{hhhhh} with values~\eqref{asdf} is a fixed point of~\eqref{hbsp}. Let us call this limit sequence $\{x_k^\star\}_{k \ge 0}$. In order to show that the limit cycle is attractive (nearby trajectories will eventually converge to the cycle) consider a perturbed version of this sequence $\{x_k^\star+\epsilon_k\}_{k \ge 0}$. If we assume that the $k^\text{th}$ iterate still belongs to the same piece of the function (e.g. if $x^\star_k < 1$ then $x^\star_k+\epsilon_k < 1$, and if $x^\star_k > 2$ then $x^\star_k+\epsilon_k >2$) then we can use the Heavy-ball equations to compute the perturbation in the subsequent iterate. Upon doing so, we find that $\{\epsilon_k\}_{k \ge 0}$ must satisfy
\[
\bmat{\epsilon_{k+2}\\\epsilon_{k+1}} = \underbrace{\bmat{ -\tfrac{4}{3} & -\tfrac{4}{9} \\ 1 & 0 }}_P \bmat{ \epsilon_{k+1} \\ \epsilon_{k} }
\qquad\text{for all }k
\]
It is immediate that $\rho(P) = \tfrac23 < 1$ so as long as no transient value of $\epsilon_k$ strays too far from zero and causes an unscheduled crossing of the dotted lines on Figure~\ref{fig:hb_counterexample}, then we will have $\epsilon_k\to 0$ and the limit cycle will be attractive. Undesired transient behavior can be ruled out by ensuring that the error eventually decreases monotonically. One can easily verify that
\[
\bl\|P^8\br\|^2 \approx 0.46044 < \frac12
\]
where $\|\cdot\|$ is the induced 2-norm. Therefore, $P^8$ is a contraction, and we have:
\begin{equation}\label{qwer}
\epsilon_{k+8}^2 
	\le \bbbl\|\bmat{\epsilon_{k+9}\\\epsilon_{k+8}}\bbbr\|^2
	\le \bl\|P^8\br\|^2 \bbbl\|\bmat{\epsilon_{k+1}\\\epsilon_{k}}\bbbr\|^2
	< \frac{1}{2}(\epsilon_{k+1}^2 + \epsilon_k^2)
\end{equation}
If eight consecutive perturbations $\{\epsilon_i^2,\epsilon_{i+1}^2,\dots,\epsilon_{i+7}^2\}$ are each less than some $\bar\epsilon^2$, then apply~\eqref{qwer} twice to conclude that
\[
\epsilon_{i+8}^2 < \frac{1}{2}(\epsilon_i^2 + \epsilon_{i+1}^2) < \bar \epsilon^2
\qquad\text{and}\qquad
\epsilon_{i+9}^2 < \frac{1}{2}(\epsilon_{i+1}^2 + \epsilon_{i+2}^2) < \bar \epsilon^2
\]
Continuing in this fashion, we conclude that the entire tail $\{\epsilon_k^2\}_{k\ge i}$ also satisfies the bound $\epsilon_k^2 < \bar\epsilon^2$. The closest that our proposed limit cycle comes to a transition point of $f(x)$ (either $1$ or $2$) is $r-2 = \tfrac{142}{1225} \approx 0.1159$. Therefore, if we set this number to be $\bar\epsilon$, and we can find eight consecutive iterates of the Heavy-ball method that are each within $\bar\epsilon$ of the limit cycle, then the remaining iterates must converge exponentially to the limit cycle.
It is straightforward to check that if $x_0=3.3$, then the iterates $x_4,x_5,\dots,x_{11}$ are each within a distance $\bar \epsilon$ of their respective limit points. Therefore, the limit cycle is attractive.

\end{document}